\documentclass[12pt]{amsart}
\usepackage{mathscinet}
\usepackage[shortlabels]{enumitem}
\usepackage{tikz}
\usetikzlibrary{arrows}
\usepackage{amsmath,amsfonts,amssymb,amsthm,epsfig,epstopdf,array,comment}
\usepackage{wasysym}
\usepackage{scalerel}
\usepackage{graphicx}
\usepackage{float}
\usepackage{centernot}
\usepackage[colorlinks=true,pagebackref,hyperindex]{hyperref}

\newcommand{\ind}{\nolinebreak{\perp \!\!\! \perp}}
\newcommand{\Hom}{\mathrm{Hom}}
\newcommand{\rank}{\mathrm{rank}}
\newcommand{\Mat}{\mathrm{Mat}}

\newcommand{\Scan}{\mathrm{Scan}}
\newcommand{\G}{\mathcal{G}}
\newcommand{\Sub}{\mathrm{\hat{S}ub}}
\newcommand{\Mixt}{\mathrm{Mixt}}
\newcommand{\minor}{\mathrm{minor}}
\newcommand{\sgn}{\mathrm{sgn}}
\newcommand{\Sym}{\mathrm{Sym}}
\newcommand{\LCM}{\mathrm{lcm}}
\newcommand{\LM}{\mathrm{LM}}
\newcommand{\LT}{\mathrm{LT}}
\newcommand{\jh}{\mathrm{\hat{j}}}
\newcommand{\kh}{\mathrm{\hat{k}}}

\theoremstyle{plain}
\newtheorem{thm}{Theorem}[section]
\newtheorem{lem}[thm]{Lemma}
\newtheorem{prop}[thm]{Proposition}
\newtheorem{cor}[thm]{Corollary}

\theoremstyle{definition}
\newtheorem{defn}[thm]{Definition}
\newtheorem{exmp}[thm]{Example}
\newtheorem{rem}[thm]{Remark}

\begin{document}

\title{Gr\"obner bases for bipartite determinantal ideals}
\author{Josua Illian}
\address{Department of Mathematics
and Statistics,
Oakland University, 
Rochester, MI 48309-4479, USA 
}
\email{jillian@oakland.edu}

\author{Li Li}
\address{Department of Mathematics
and Statistics,
Oakland University, 
Rochester, MI 48309-4479, USA 
}
\email{li2345@oakland.edu}


\begin{abstract}
Nakajima's graded quiver varieties naturally appear in the study of bases of cluster algebras. One particular family of these varieties, namely the \emph{bipartite determinantal varieties}, can be defined for any bipartite quiver and gives a vast generalization of classical determinantal varieties with broad applications to algebra, geometry, combinatorics, and statistics. The ideals that define bipartite determinantal varieties are called bipartite determinantal ideals.
We provide an elementary method of proof showing that the natural generators of a bipartite determinantal ideal form a Gr\"obner basis, using an S-polynomial construction method that relies on the Leibniz formula for determinants. This method is developed from an idea by Narasimhan \cite{Narasimhan} and Caniglia--Guccione--Guccione \cite{Caniglia}.
As applications, we study the connection between double determinantal ideals (which are bipartite determinantal ideals of a quiver with two vertices) and tensors, and provide an interpretation of these ideals within the context of algebraic statistics.
\end{abstract}

\maketitle


\section{Introduction}
The main objective of the paper is to study bipartite determinantal ideals and in particular, double determinantal ideals. We shall prove that the natural generators of such an ideal form a Gr\"obner basis under a suitable term order. Additionally, we study the connection between double determinantal ideals and the 3-dimensional tensors as well as algebraic statistics. 

We start by recalling the classical determinantal ideals. Let $K$ be a field, $C$ be an $m\times n$ matrix of independent variables $x_{ij}$, let $u$ be a positive integer. The ideal $I^{\rm det}_{m,n,u}$ of $K[x_{ij}]$ generated by all $u\times u$ minors in $A$ is called a determinantal ideal, and $K[x_{ij}]/I^{\rm det}_{m,n,u}$ is called a determinantal ring, 
the affine variety defined by a determinantal ideal is called a determinantal variety.  
Determinantal ideals have long been a central topic, a test field, and a source of inspiration for various fields including commutative algebra, algebraic geometry, representation theory, and combinatorics.
Through the use of Gr\"obner basis theory and simplicial complexes, researchers have studied the minimal free resolution and syzygies (or in geometric terms of vector bundles) of determinantal ideals, as well as their Arithmatic Cohen-Macaulayness, multiplicities, powers and products, Hilbert series, $a$-invariant, $F$-regularity and $F$-rationality, local cohomology, etc. Various tools such as the straightening law and the Knuth-Robinson-Schensted (KRS) correspondence, non-intersecting paths from algebraic combinatorics, and liaison theory from commutative algebra have been employed in these studies. Determinantal ideals are closely related to the study of Grassmannian and Schubert varieties, and have many interesting specializations, generalizations, and variations, including Segre varieties, linear determinantal varieties, rational normal scrolls, symmetric and skew-symmetric determinantal varieties, (classical, two-sided, or mixed) ladder determinantal varieties, etc. For more information, readers may referred to \cite{Baetica, BC, Harris, MS, Miro-Roig, Weyman} and the references therein.

This paper studies bipartite determinantal ideals which are closely related to Nakajima's graded quiver varieties. Indeed, for a given bipartite quiver, the corresponding Nakajima's affine graded quiver varieties are exactly the affine varieties defined by the bipartite determinantal ideals  (see \ref{section:bipartite det}). 

Recall that a bipartite quiver is a quiver $\mathcal{Q}$ with vertex set ${\rm V}_\mathcal{Q}={\rm V}_{\rm source}\sqcup {\rm V}_{\rm sink}$ and each arrow $h$ has source ${\rm s}(h)\in {\rm V}_{\rm source}$ and target ${\rm t}(h)\in {\rm V}_{\rm sink}$.
Given a matrix $A$ and a positive integer $u$, denote $D_{u}(A)$ to be the set $\{u\text{-minors of }A \}$.

\begin{defn}\label{bipartitedeterminantalideal}
 Let $\mathcal{Q}$ be a bipartite quiver with $d$ vertices and $r$ arrows $h_i:{\rm s}(h_i)\to {\rm t}(h_i)$ (for $i=1,\dots,r$).
 Let ${\bf m}=(m_1,\dots,m_d)$, ${\bf u}=(u_1,\dots,v_d)$  be $d$-tuples of nonnegative integers.
Let $X= \{ X^{(k)}=[x^{(k)}_{ij}]\}_{k=1}^r $ be a set of variable matrices where the size of $X^{(k)}$ is $m_{{\rm t}(h_k)}\times m_{{\rm s}(h_k)}$.  
For $\alpha\in {\rm V}_{\rm sink}$, let $r_1<\cdots<r_s$ be the indices such that $h_{r_1},\dots,h_{r_s}$ are all the arrows with target $\alpha$, and define
	$$ A_\alpha= [ X^{(r_1)}|X^{(r_2)}|\hdots|X^{(r_s)} ].$$
For $\beta\in {\rm V}_{\rm source}$, let $r'_1<\cdots<r'_t$ be the indices such that $h_{r'_1},\dots,h_{r'_t}$ are all the arrows with source $\beta$, and define
	$$A_\beta=\left[
	\begin{array}{c}
	X^{(r'_1)}  \\ \hline
	X^{(r'_2)} \\ \hline
	\vdots \\ \hline
	X^{(r'_t)}
	\end{array}
	\right]. $$	
For $\gamma\in{\rm V}_\mathcal{Q}$, let $a_\gamma\times b_\gamma$ be the size of $A_\gamma$, that is, 
$(a_\alpha,b_\alpha) :=(m_\alpha,\sum_{h:\ {\rm t}(h)=\alpha} m_{{\rm s}(h)})$ for $\alpha\in {\rm V}_{\rm sink}$, 
$(a_\beta,b_\beta) :=(\sum_{h:\ {\rm s}(h)=\beta} m_\alpha,m_{{\rm t}(h)})$ for $\beta\in {\rm V}_{\rm source}$. 
The \textit{bipartite determinantal ideal} $I_{\mathcal{Q},{\bf m},{\bf u}}$ is the ideal of $K[X]$ generated by $\bigcup_{\gamma\in {\rm V}_\mathcal{Q}}D_{u_\gamma+1}(A_\gamma)$.  These generators are called the \emph{natural generators} of $I_{\mathcal{Q},{\bf m},{\bf u}}$. 
In particular, we call $I_{\mathcal{Q},{\bf m},{\bf u}}$ a double determinantal ideal when $Q$ has two vertices.  
\end{defn}
Obviously, 
$\{$determinantal ideals$\}\subseteq \{$double determinantal ideals$\}\subseteq \{$bipartite determinantal ideals$\}$. 
Gr\"obner theory plays a fundamental role in the study of the classical determinantal ideals. A key fact is that the minors that generate a determinantal ideal form a Gr\"obner basis (as remarked in \cite{BC}, this fact was first proved by Narasimhan \cite{Narasimhan} and then reproved many times).  
The main result of this paper is a similar statement for bipartite determinantal ideals.

\begin{thm} \label{theorem:bipartite determinantal}
The natural generators of  
{bipartite determinantal ideal} $I_{\mathcal{Q},{\bf m},{\bf u}}$ forms a Gr\"obner basis with respect to any lexicographical monomial order that is consistent in $A_\gamma$ for all $\gamma\in{\rm V}_\mathcal{Q}$ in the sense of Definition \ref{order}.
\end{thm}

\begin{exmp}
An example of a double determinantal variety is given in Example \ref{eg:m=n=r=u=v=2}. For a more general example,
consider the following quiver:
\begin{figure}[h]
\begin{center}
	\begin{tikzpicture}[scale=1.3]
	\draw (0,0)node{$3$};
	\draw (2,0)node{$1$};
	\draw (0,-1.5)node{$4$};
	\draw (2,-1.5)node{$2$};
	
	\draw[-stealth](2-.3,0.2) -- (.3,0.2)  node [midway] {\tiny$h_1$};
	\draw[-stealth](2-.3,0) -- (.3,0) node [midway] {\tiny$h_2$};
	\draw[-stealth](2-.3,-.2) -- (.3,-.2) node [midway] {\tiny$h_3$};

	\draw[-stealth](2-.3,-.3) -- (.3,-1.5+.3)  node [pos=.4,above] {\tiny$h_4$};
	\draw[-stealth](2-.3+.1,-.3-.1) -- (.3+.1,-1.5+.3-.1)  node [pos=.2,below] {\tiny$h_5$};

	\draw[-stealth](2-.3,-1.5+.3) -- (.3,-.3)  node [pos=.3,below] {\tiny$h_6$};

	\draw[-stealth](2-.3,-1.4) -- (.3,-1.4)  node [midway] {\tiny$h_7$};
	\draw[-stealth](2-.3,-1.6) -- (.3,-1.6) node [midway] {\tiny$h_8$};
	\end{tikzpicture} 
\label{fig:bipartite}
\end{center}
\end{figure}

Let $(m_1,m_2,m_3,m_4)=(4,4,3,3)$, $u_1=\cdots=u_8=1$. 
Then all $X^{(i)}$ have size $3\times 4$, and 
$$
A_1
= \left[
	\begin{array}{c}
	X^{(1)}  \\ \hline
	X^{(2)} \\ \hline
	X^{(3)} \\ \hline
	X^{(4)} \\ \hline
	X^{(5)} \\ 
	\end{array}
	\right]
,\,
A_2
= \left[
	\begin{array}{c}
	X^{(6)}  \\ \hline
	X^{(7)} \\ \hline
	X^{(8)} \\ 
	\end{array}
	\right]
,\,
A_3=[X^{(1)}|X^{(2)}|X^{(3)}|X^{(6)}]
,\,
A_4=[X^{(4)}|X^{(5)}|X^{(7)}|X^{(8)}]
$$
Theorem \ref{theorem:bipartite determinantal} asserts that the all the $2\times2$-minors in $A_1,\dots,A_4$ form a Gr\"obner basis with respect to any lexicographical order that is consistent in all $A_i$ (as an example of such an order, take the reading order on each $X^{(k)}$, and let the variables in $X^{(k)}$ to be larger than the variables in $X^{(k')}$ if $k<k'$). 
\end{exmp}

Note that Fieldsteel and Klein have proved Theorem \ref{theorem:bipartite determinantal} for double determinantal ideals in \cite{FK}, and their method probably extends to all bipartite determinantal ideals. Their elegant proof uses an advanced tool in commutative algebra called the liaison theory, which is very effective in the study of the determinantal varieties and the various generalizations. 
In contrast, our proof is more elementary and uses only the  S-pair argument described in Proposition \ref{chain};
this idea is inspired by the work of Narasimhan \cite{Narasimhan} and Caniglia--Guccione--Guccione \cite{Caniglia}. 
In \cite{Narasimhan}, arising from Abyankar's work on singularities of Schubert varieties of flag manifolds, Narasimhan established the primality, and thus irreducibility, of the ladder determinantal ideal. 
%
%
We would also like to mention that Conca, De Negri, and Stojanac proved in \cite{Conca} the Gr\"obner basis result on the generators of two flattenings of an order 3 tensor of size $2\times a\times b$, which is a special case of the double determinantal ideals studied in this paper.

The paper is organized as follows. 
In \S2 we recall the background of Gr\"obner bases, give a key example (Example \ref{example1}), and we prove the main theorem on Gr\"obner bases.
In \S3 we discuss some applications by relating the double determinantal ideals to the 3-dimensional tensors and algebraic statistics. 
Most results of the paper were originally included in the first author's PhD dissertation in 2020.  
\smallskip

{\bf Acknowledgments.}  We gratefully acknowledge discussions with Patricia Klein, Allen Knutson, and Alex Yong; we thank Bernd Sturmfels and Seth Sullivant for pointing to us the relations of the double determinantal varieties to tensors and algebraic statistics; we thank  Nathan Fieldsteel and Patricia Klein for sharing drafts of their papers. We greatly appreciate the referee for carefully reading through the paper and providing many helpful comments and suggestions.   Computer calculations were performed using Macaulay 2 \cite{M2}.

\section{Gr\"obner bases}

\subsection{Gr\"obner bases and monomial orders}\label{subsection:facts about Groebner basis}
Let $R=K[x_1,\hdots, x_n]$ be a polynomial ring in a finite number of variables over a field $K$.  In this paper, the term \emph{monomial} refers to a product of the form $x_1^{i_1}\cdots x_n^{i_n}$ with coefficient 1. Let $\mathcal{N}$ be the set of all monomials in $R$ with a chosen monomial order ``$>$''; for a polynomial $f\in R$, let  $\LM(f)=\LM_{>}(f)$ be the leading monomial of $f$, $\LT(f)=\LT_{>}(f)$ be the leading term of $f$, $\LT(I)=\LT_{>}(I)$ be the initial ideal of $I$.
We say that a finite set $\G$ is a \textit{Gr\"obner basis}  (for the  ideal $I$ generated by $\G$) with respect to the monomial order ``$>$''  if $\LT_{>}(I)$ is generated by the set $ \{ \LT_{>}(f)|f\in \G \} $.
Given any two polynomials $f,g$, their \textit{S-pair} is defined as
	\[ S(f,g) = \frac{L}{\LT(f)}f - \frac{L}{\LT(g)}g, \textrm{ where  }L=\LCM(\LM(f),\LM(g))\in\mathcal{N}. \]
The well-known Buchberger's criterion says that $\G$ is a Gr\"obner basis if for any $f,g\in \G$,  the remainder of $S(f,g)$ by $\G$ is zero; 
the following variation will be used in our proof.

\begin{prop} [Theorem 3.2 in \cite{Moeller}; also see \cite{Caniglia}] \label{chain}
	A set $\G$  of polynomials is a Gr\"obner basis (for the ideal it generates) if and only if for any two polynomials $M,N\in \G$ there exists a finite chain of polynomials $M_0, M_1,...,M_{k-1},M_k \in\G$ 
	such that $M_0=M$, $M_k=N$, and the following holds for all $i=1,\hdots,k$:
	
	{\rm(1)} $\LM(M_i) | \LCM(\LM(M),\LM(N))$, and
	
	{\rm(2)} $S(M_{i-1},M_i)=\sum_{j=1}^n a_jP_j$, where $n\in \mathbb{Z}_{\ge0}$, $a_j$ are monomials, $P_j\in \G$, and $$\LM(a_jP_j)<\LCM(\LM(M_{i-1}),\LM(M_i)) \textrm{ for all }j .$$
\end{prop}

\begin{defn}\label{order}
Let $C=(x_{ij})$ be a matrix of variables and ``$>$'' be a monomial order.  
We say that  ``$>$'' is \textit{consistent} with $C$ if $x_{ij}> x_{ij'}$ whenever $j'>j$, and $x_{ij}> x_{i'j}$ whenever $i'>i$; that is, the variables decrease from left to right and from top to bottom.
We say that two points $(i,j)$, $(k,l)$ are in NW-SE position (or the entries $x_{ij}$ and $x_{kl}$ of a matrix are in NW-SE position) when $(k-i)(l-j)>0$. We say that a sequence of points (or entries of a matrix) are in NW-SE position if every pair is in NW-SE position.
\end{defn}

Note that a lexicographic order consistent with $C$ is a \textit{diagonal order}, that is, for every square submatrix $D$ of $C$, the leading term of $\det(D)$ is the product of the main diagonal entries of $D$. 
An example of a consistent order is the reading order, namely the lexicographic order with  
$ x_{11}> x_{12}>\hdots > x_{1n} > x_{21}>x_{22}>\hdots > x_{2n} > \cdots>
x_{m1}>x_{m2}>\hdots > x_{mn}$.
In the setup in Definition \ref{bipartitedeterminantalideal}, there exists a lexicographical order that is consistent with $A_\gamma$ for all $\gamma$. For example,  define  ``$>$'' by requiring that  $x_{ij}^{(s)}> x_{kl}^{(t)}$ when   $s<t$, and use the reading order if $s=t$. 
Under such an order, the leading term of any minor $M$ of $A_\gamma$  consists of a sequence of variables of $A_\gamma$ which are in NW-SE position.

\subsection{Notations and a key example}\label{lincom}
	
Recall the Leibniz formula for matrix determinants: if $C=\left[ x_{ij} \right]$ is an $n\times n$ matrix of variables, then
	\begin{equation}\label{leibniz}
		\det(C) = \sum_{\sigma\in S_n} \sgn(\sigma) \prod_{i=1}^n x_{\sigma(i),i} =\sum_{\tau\in S_n}\sgn(\tau) \prod_{i=1}^n x_{i,\tau(i)}. 
	\end{equation}

	\begin{defn} \label{def}
	
	 We think the $r$ matrices $X^{(1)},\dots, X^{(r)}$ as ``pages'' in a ``book". The $k$-th page contains $X^{(k)}$. 
We use the following (unusual) coordinate system to be consistent with the labelling of entries in a matrix:
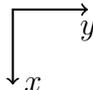
\begin{center}
\begin{tikzpicture}[scale=1]
    \draw [<->,thick] (0,-1) node (yaxis) [right] {$x$}
        |- (1,0) node (xaxis) [below] {$y$};
\end{tikzpicture}
\end{center}

Define 
$$\mathcal{L}:=\{(i,j,k)\in\mathbb{Z}^3 \ | \  1\le k\le r, (i,j)\in [1,m_{{\rm t}(h_k)}]_\mathbb{Z}\times [1,m_{{\rm s}(h_k)}]_\mathbb{Z}\}. $$
For $\gamma\in {\rm V}_\mathcal{Q}$, 
define the following injective map (recall that $a_\gamma\times b_\gamma$ is the size of $A_\gamma$) 
$$\phi_{\gamma}: [1,a_\gamma]_\mathbb{Z}\times[1,b_\gamma]_\mathbb{Z}\to \mathcal{L}, \quad 
(p,q)\mapsto (i,j,k) \textrm{ if the $(p,q)$-entry of $A_\gamma$ is $x_{ij}^{(k)}$}.$$ 

    Denote $x_p=x_{ij}^{(k)}$ (if well-defined) for a lattice point $p=(i,j,k)\in\mathbb{Z}^3$.
    
		Given an arrow $\gamma_2\to \gamma_1$ in $\mathcal{Q}$, $u,v\in\mathbb{Z}_{>0}$, denote $A=A_{\gamma_1}$, $B=A_{\gamma_2}$. For $M\in D_{u}(A),N\in D_{v}(B)$,  define
		\[ L=\LCM(\LM(M),\LM(N))=x^{(r_1)}_{\alpha_1\beta_1}x^{(r_2)}_{\alpha_2\beta_2}\hdots x^{(r_l)}_{\alpha_l\beta_l}\]
		where $l$ is the degree of $L$, and the right side satisfies $x_{\alpha_i \beta_i}^{(r_i)} > x_{ \alpha_{i+1} \beta_{i+1} }^{(r_{i+1})}$ for all $i=1,...,l-1$.  
		Denote $p_i=(\alpha_i,\beta_i,r_i)\in\mathbb{Z}^3$ ( so $x_{p_i}=x_{\alpha_i \beta_i}^{(r_i)}$ and $L=\prod_{i=1}^l x_{p_i}$ ). 
Define 
		\[ S_M= \{ i: x_{p_i}\textrm{ divides }\LM(M)\},
		\quad
		 S_N=\{ i: x_{p_i} \textrm{ divides } \LM(N)\}, \textrm{ as subsets of $\{ 1,\hdots,l\}$}. \]
Let $\Sym_l$ be the symmetric group of $\{1,\dots,l\}$, and define its subgroups
$$\aligned
&\Sym(S_M)=\{\sigma\in \Sym_l:\sigma(i)=i \text{ for each } i\notin S_M \}, \\
&\Sym(S_N)=\{\tau\in \Sym_l:\tau(i)=i \text{ for each } i\notin S_N \}. \\
\endaligned$$
Define $G:=\Sym(S_M)\times \Sym(S_N)$. 
For $(\sigma,\tau)\in G$, define  
$$(\sigma,\tau)p_i=(\alpha_{\sigma(i)},\beta_{\tau(i)},r_i) \textrm{ for all }i=1,\hdots, l.$$ 
We say that $p_i$ is a \textit{fixed point} of $(\sigma,\tau)$ if $(\sigma,\tau)p_i=p_i$, or equivalently, 
if $\sigma(i)=i$ and  $\tau(i)=i$.  
Define 
		\[ L(\sigma,\tau) = \prod_{i=1}^l x_{(\sigma,\tau)p_i} = x^{(r_1)}_{\alpha_{\sigma(1)},\beta_{\tau(1)}}x^{(r_2)}_{\alpha_{\sigma(2)},\beta_{\tau(2)}}\hdots x^{(r_l)}_{\alpha_{\sigma(l)},\beta_{\tau(l)}} . \]
	In particular, $L(1,1)=L$. Note that $L(\sigma,\tau)$ may be non-squarefree.

For each $\sigma \in \Sym(S_M)$, define a polynomial
		$\displaystyle P_{\sigma,\cdot} = \sum_{\tau\in \Sym(S_N)}\sgn(\sigma)\sgn(\tau)L(\sigma,\tau).$
		
For each $\tau \in \Sym(S_N)$, define a polynomial
    	        $\displaystyle P_{\cdot,\tau} = \sum_{\sigma\in \Sym(S_M)}\sgn(\sigma)\sgn(\tau)L(\sigma,\tau).$
	\end{defn}

The following is a simple but crucial identity:  
	\begin{equation}\label{doublesum}
	\begin{split}
		\sum_{\sigma\in \Sym(S_M)}P_{\sigma,\cdot} &= \sum_{\sigma\in \Sym(S_M)} \sum_{\tau\in \Sym(S_N)} \sgn(\sigma)\sgn(\tau)L(\sigma,\tau) \\
		&= \sum_{\tau\in \Sym(S_N)} \sum_{\sigma\in \Sym(S_M)} \sgn(\sigma)\sgn(\tau)L(\sigma,\tau)
		= \sum_{\tau\in \Sym(S_N)}P_{\cdot,\tau}.
	\end{split}
	\end{equation}

\begin{exmp}\label{example1} (A key example)
Let $\mathcal{Q}:1{\leftarrow}2$, ${\bf m}=(3,3)$,  ${\bf u}=(1,1)$, and write $x^{(1)}_{ij}$ as $x_{ij}$.  
Let $M=\begin{vmatrix}x_{21}&x_{23}\\x_{31}&x_{33}\end{vmatrix}$,
$N=\begin{vmatrix}x_{12}&x_{13}\\x_{32}&x_{33}\end{vmatrix}$.
Take the order ``$>$'' where $x_{ij}>x_{kl}$ when $i<k$ or when $i=k$ and $j<l$.  
We use diagrams as a convenient way to represent $L$ and the newly created monomials.  The variables of $\LT(M)$ are represented by ``$\Circle$'' and are attached to horizontal rays on the left, while variables of $\LT(N)$ are represented by ``$\times$'' and are attached to vertical rays on the top; $\Sym(S_M)$ permutes the horizontal rays and $\Sym(S_N)$ permutes the vertical rays.  Then
$$\aligned
			P_{1,\cdot} &  =  L -L(1,(13))=  x_{21}N, &
			P_{(23),\cdot} & =  -L((23),1) + L((23),(13))= -x_{31} \left| \begin{matrix}
				x_{12} & x_{13} \\
				x_{22} & x_{23} \\
			\end{matrix} \right|, \\
\endaligned$$
$$\aligned			
			P_{\cdot,1} & = L -L((23),1)=  x_{12}M, &
			P_{\cdot,(13)} & = -L(1,(13)) + L((23),(13)) = -x_{13} \left| \begin{matrix}
				x_{21} & x_{22} \\
				x_{31} & x_{32} \\
			\end{matrix} \right|.  
\endaligned$$
See Figure \ref{fig:example1b}. 
\begin{figure}[h]
\begin{center}
\begin{tikzpicture}[scale=0.7]
			\draw (0,0) node[]{$\cdot$};
			\draw (0,-1) circle (1.75mm);
			\draw (2,0) node[]{$\cdot$};
			\draw (1,0) node[]{$\times$};
			\draw (1,-1) node[]{$\cdot$};
			\draw (2,-1) node[]{$\cdot$};
			\draw (0,-2) node[]{$\cdot$};
			\draw (1,-2) node[]{$\cdot$};
			\draw (2,-2) node[]{$\otimes$};
			\draw(-.5,-1)--(-.17,-1);
			\draw(1,0)--(1,.5);
			\draw (-.5,-2) -- (1.83,-2);
			\draw (2,.5)--(2,-1.83);
			\draw (1,-2.75) node[]{\small$L(1,1)=x_{12}x_{21}x_{33}$};
	\draw (1.3,0.3) node[]{\tiny{1}};
	\draw (0.3,-0.7) node[]{\tiny{2}};
	\draw (2.3,-1.7) node[]{\tiny{3}};
			\begin{scope}[shift={(6,0)}]
			\draw (0,0) node[]{$\cdot$};
			\draw (0,-2) circle (1.75mm);
			\draw (2,0) node[]{$\cdot$};
			\draw (1,0) node[]{$\times$};
			\draw (1,-1) node[]{$\cdot$};
			\draw (0,-1) node[]{$\cdot$};
			\draw (2,-1) node[]{$\cdot$};
			\draw (2,-2) node[]{$\cdot$};
			\draw (1,-2) node[]{$\cdot$};
			\draw (2,-1) node[]{$\otimes$};
			\draw(-.5,-2)--(-.17,-2);
			\draw(1,0)--(1,.5);
			\draw (-.5,-1) -- (1.83,-1);
			\draw (2,.5)--(2,-0.83);
			\draw (1,-2.75) node[]{\small$L((23),1)=x_{12}x_{31}x_{23}$};
	\draw (1.3,0.3) node[]{\tiny{1}};
	\draw (0.3,-1.7) node[]{\tiny{2}};
	\draw (2.3,-0.7) node[]{\tiny{3}};
			\end{scope}
			\begin{scope}[shift={(12,0)}]			
			\draw (0,0) node[]{$\cdot$};
			\draw (0,-1) circle (1.75mm);
			\draw (1,0) node[]{$\cdot$};
			\draw (2,0) node[]{$\times$};
			\draw (1,-1) node[]{$\cdot$};
			\draw (2,-1) node[]{$\cdot$};
			\draw (0,-2) node[]{$\cdot$};
			\draw (2,-2) node[]{$\cdot$};
			\draw (1,-2) node[]{$\otimes$};
			\draw(-.5,-1)--(-.17,-1);
			\draw(2,0)--(2,.5);
			\draw (-.5,-2) -- (0.83,-2);
			\draw (1,.5)--(1,-1.83);
			\draw (1,-2.75) node[]{\small$L(1,(13))=x_{13}x_{21}x_{32}$};
	\draw (2.3,0.3) node[]{\tiny{1}};
	\draw (0.3,-0.7) node[]{\tiny{2}};
	\draw (1.3,-1.7) node[]{\tiny{3}};
			\end{scope}
			\begin{scope}[shift={(18,0)}]
			\draw (0,0) node[]{$\cdot$};
			\draw (0,-2) circle (1.75mm);
			\draw (1,0) node[]{$\cdot$};
			\draw (2,0) node[]{$\times$};
			\draw (0,-1) node[]{$\cdot$};
			\draw (2,-1) node[]{$\cdot$};
			\draw (1,-2) node[]{$\cdot$};
			\draw (2,-2) node[]{$\cdot$};
			\draw (1,-1) node[]{$\otimes$};
			\draw(-.5,-1)--(0.83,-1);
			\draw(2,0)--(2,.5);
			\draw (-.5,-2) -- (-0.17,-2);
			\draw (1,.5)--(1,-0.83);
			\draw (1.5,-2.85) node[]{\small$L((23),(13))=x_{13}x_{31}x_{22}$};
	\draw (2.3,0.3) node[]{\tiny{1}};
	\draw (1.3,-0.7) node[]{\tiny{3}};
	\draw (0.3,-1.7) node[]{\tiny{2}};
			\end{scope}
			\end{tikzpicture}
\end{center}		
			\caption{Example \ref{example1}}
			\label{fig:example1b}
		\end{figure}
		
It follows from \eqref{doublesum}  that $P_{1,\cdot}+P_{(23),\cdot}= 	P_{\cdot,1}+P_{\cdot,(13)}$, so
$$S(M,N)=P_{\cdot,1}-P_{1,\cdot}=P_{(23),\cdot}-P_{\cdot,(13)}=
-x_{31} \left| \begin{matrix}
				x_{12} & x_{13} \\
				x_{22} & x_{23} \\
			\end{matrix} \right| 			
+x_{13} \left| \begin{matrix}
				x_{21} & x_{22} \\
				x_{31} & x_{32} \\
			\end{matrix} \right|
			$$
Note that the leading monomials of the two terms on the right side are both $<L$, so the condition (2) of Proposition \ref{chain} holds (with $M_0=M, M_1=N, i=1$). 
\end{exmp}
	
\subsection{Pseudominors}	
\begin{defn} \label{pseudominor}
		Given $\gamma\in {\rm V}_\mathcal{Q}$, and two lists (with possible repetition) $a_1,\dots,a_p\in\{1,\dots,m\}$, $b_1,\dots,b_q\in\{1,\dots,n\}$, we define a \textit{pseudosubmatrix} $A_\gamma(a_1,\dots,a_p;b_1,\dots,b_q)$ to be a matrix of size $p\times q$ whose $(i,j)$-entry is the $(a_i,b_j)$-entry of $A_\gamma$.  Its determinant is called a \textit{pseudominor}, and is denoted 
	\[ \minor_{A_\gamma}(a_1,\hdots,a_u;b_1,\hdots,b_u)
= \left|  \begin{matrix}
	x_{\phi_\gamma(a_1,b_1)} & \hdotsfor{1} & x_{\phi_\gamma(a_1,b_u)} \\
	\vdots & \ddots & \vdots \\
	x_{\phi_\gamma(a_u,b_1)} & \hdotsfor{1} & x_{\phi_\gamma(a_u,b_u)}
	\end{matrix} \right| 
 \]
 We call a pseudosubmatrix (and the corresponding pseudominor) \textit{trivial} if it has repeated rows or columns, otherwise we call it \textit{non-trivial}. 
	\end{defn}   
	
\begin{prop}\label{minor}
		Use the notation in Definition \ref{def} and denote $S_M=\{i_1,\hdots,i_u  \}$ and $S_N=\{j_1,\hdots,j_v \}$ where $i_1<i_2<\hdots <i_u$ and $j_1<j_2<\hdots <j_v$.
		  Then for $\tau \in \Sym(S_N)$, 
		\begin{align*}
			P_{\cdot,\tau} = \bigg( \sgn(\tau)\prod_{k\in S_N\setminus S_M} x^{(r_{k})}_{\alpha_{k},\beta_{\tau(k)}} \bigg) \left|  \begin{matrix}
				x^{(r_{i_1})}_{\alpha_{i_1},\beta_{\tau(i_1)}} & \hdotsfor{1} & x^{(r_{i_u})}_{\alpha_{i_1},\beta_{\tau(i_u)}} \\
				\vdots & \ddots & \vdots \\
				x^{(r_{i_1})}_{\alpha_{i_u},\beta_{\tau(i_1)}} & \hdotsfor{1} & x^{(r_{i_u})}_{\alpha_{i_u},\beta_{\tau(i_u)}}
			\end{matrix}
			\right| 
		\end{align*}
where the determinant is the pseudominor 
 $A(\alpha_{i_1},\dots,\alpha_{i_u};b_1,\dots,b_u)$ where $b_t$ (for $1\le t\le u$) are determined by $\phi_{\gamma_1}(*,b_t)=(*,\beta_{\tau(i_t)},r_{i_t})$. 
		  Similarly for $\sigma \in \Sym(S_M)$, 
		\begin{align*}
			P_{\sigma,\cdot} =  \bigg( \sgn(\sigma)\prod_{k\in S_M\setminus S_N} x^{(r_{k})}_{\alpha_{\sigma(k)},\beta_{k}} \bigg) 
			 \begin{vmatrix}
				x^{(r_{j_1})}_{\alpha_{\sigma(j_1)},\beta_{j_1}} & \hdotsfor{1} & x^{(r_{j_1})}_{\alpha_{\sigma(j_1)},\beta_{j_v}} \\
				\vdots & \ddots & \vdots \\
				x^{(r_{j_v})}_{\alpha_{\sigma(j_v)},\beta_{j_1}} & \hdotsfor{1} & x^{(r_{j_v})}_{\alpha_{\sigma(j_v)},\beta_{j_v}}
			\end{vmatrix}.   
		\end{align*}
where the determinant is the pseudominor $B(a_1,\dots,a_v;\beta_{j_1},\dots,\beta_{j_v})$ 
where $a_t$ (for $1\le t\le v$) are determined by $\phi_{\gamma_2}(a_t,*)=(\alpha_{j_t},*,r_{j_t})$. 
In particular, 	\[ P_{\cdot,1} = \dfrac{L}{\LT(M)} M\text{ and }P_{1,\cdot} =  \dfrac{L}{\LT(N)} N.\]
\end{prop}
	
	\begin{proof} By symmetry, we only prove the identity of $P_{\cdot,\tau}$.
 Fix $\tau\in \Sym(S_N)$.  Then 
$$P_{\cdot,\tau}
=  \sum_{\sigma\in \Sym(S_M)}\sgn(\sigma) \sgn(\tau) \prod_{k=1}^l x_{(\sigma,\tau)p_k} 
= \bigg(\prod
		_{k\in S_N\setminus S_M} x_{(\sigma,\tau)p_k}
		  \bigg)
		\Delta
= \bigg(\prod
		_{k\in S_N\setminus S_M} x_{(1,\tau)p_k}
		  \bigg)
		\Delta,
$$	
where 
$$\Delta:= \sum_{\sigma\in \Sym(S_M)} \sgn(\sigma ) x_{(\sigma,\tau)p_k}
=
\left|  \begin{matrix}
		x_{\alpha_{i_1},\beta_{\tau(i_1)}}^{(r_{i_1})} & \hdotsfor{1} & x_{\alpha_{i_1},\beta_{\tau(i_u)}}^{(r_{i_u})} \\
		\vdots & \ddots & \vdots \\
		x_{\alpha_{i_u},\beta_{\tau(i_1)}}^{(r_{i_1})} & \hdotsfor{1} & x_{\alpha_{i_u},\beta_{\tau(i_u)}}^{(r_{i_u})}
		\end{matrix}
		\right|
$$
by  \eqref{leibniz}. It is obvious that $\Delta$ is the pseudominor in $A$ as stated in the proposition.
\end{proof}

\subsection{Conditions for $L(\sigma,\tau)=L(\sigma',\tau')$, $P_{\sigma,\cdot}=P_{\sigma',\cdot}$ and $P_{\cdot,\tau}=P_{\cdot,\tau'}$}  
	\begin{defn} \label{incidencedef}
		Use the notations as in Definition \ref{def}.  		
		An \textit{incidence} of $M$ and $N$ is a variable $x_{p_k}$ that divides both $\LT(M)$ and $\LT(N)$, that is, $k\in S_M\cap S_N$; by abuse of terminology, we also call $p_k$ an incidence of $M$ and $N$.  Define
		
		$\Sigma_j:=\{ i: r_i=j\}$;
		 	 
		$I_j :=\{ i\in \Sigma_j: p_i \text{ is an incidence} \}$ for $j=1\hdots r$ (so $S_M\cap S_N=\bigcup_1^r I_j$);
		
		$\Sym(I_j) :=\{\sigma\in \Sym_l:\sigma(i)=i,\ \forall i\notin I_j \}\le \Sym_l$;
		
		$\prod\Sym(I_j) :=\Sym(I_1)\times \hdots \times \Sym(I_r)=\{\sigma\in\Sym(S_M\cap S_N) :  r_i=r_{\sigma(i)} \}\le \Sym_l$;
		
		$\overline{\Sym(S_M)}:=\Sym(S_M)/\prod\Sym(I_j)$ and $\overline{\Sym(S_N)}:=\Sym(S_N)/\prod\Sym(I_j)$, where the quotients are as sets of right cosets;  
		
		$H:=\{(\pi,\pi):\pi\in \prod\Sym(I_j) \}$ which is a subgroup of $G:=\Sym(S_M)\times \Sym(S_N)$,  $G/H:=\{gH\}$ is the set of right cosets of $H$ in $G$.
	\end{defn}
	
	
	\begin{lem} \label{sufficient}
		If $\overline{(\sigma,\tau)}=\overline{(\sigma',\tau')}$ in the quotient $G/H$, then $L(\sigma,\tau)=L(\sigma',\tau')$.
	\end{lem}
	\begin{proof}
Assume $\overline{(\sigma,\tau)}=\overline{(\sigma',\tau')}$. Then $\sigma=\sigma'\pi$ and $\tau=\tau'\pi$ for some $\pi\in \prod\Sym(I_j)$, 
thus
$L(\sigma,\tau) = \prod_i x_{(\sigma,\tau)p_i} = \prod_i x_{(\sigma'\pi,\tau'\pi)p_i} = \prod_i x_{(\sigma',\tau')p_{\pi(i)}} = \prod_i x_{(\sigma',\tau')p_{i}} = L(\sigma',\tau')$.
\end{proof}

	\begin{prop} \label{equalP} Let $\sigma,\sigma'\in \Sym(S_M)$ and $\tau,\tau'\in \Sym(S_N)$ be such that $P_{\sigma,\cdot}$ and $P_{\cdot,\tau}$ are non-trivial pseudominors.  Then
		\[ P_{\sigma,\cdot}=P_{\sigma',\cdot} \Leftrightarrow \overline{\sigma}=\overline{\sigma'} \,\,\mathrm{in}\,\, \overline{\Sym(S_M)},\quad
		P_{\cdot,\tau}=P_{\cdot,\tau'} \Leftrightarrow \overline{\tau}=\overline{\tau'} \,\,\mathrm{in}\,\, \overline{\Sym(S_N)}\] 
	Consequently, $P_{\overline{\sigma},\cdot}:=P_{\sigma,\cdot}$, and $P_{\cdot,\overline{\tau}}:=P_{\cdot,\tau}$ are well-defined. 
	\end{prop}
	\begin{proof}
		By symmetry, we only prove the first equivalence.
		
\noindent		``$\Leftarrow$'':  Suppose that $\overline{\sigma}=\overline{\sigma'}$. Then $\sigma'=\sigma\pi$ for some $\pi\in \prod\Sym(I_j)$, and $\pi$ induces an automorphism on $\Sym(S_N)$. We have 
		\begin{align*}
			P_{\sigma',\cdot} &= \sum_{\tau' \in \Sym(S_N)} \sgn(\sigma')\sgn(\tau')L(\sigma',\tau') 
			= \sum_{\tau \in \Sym(S_N)} \sgn(\sigma \pi)\sgn(\tau \pi)L(\sigma \pi,\tau \pi) \\
			&= \sum_{\tau \in \Sym(S_N)} \sgn(\sigma \pi)\sgn(\tau \pi)L(\sigma,\tau)= \sum_{\tau \in \Sym(S_N)} \sgn(\sigma)\sgn(\tau)L(\sigma,\tau) = P_{\sigma,\cdot}  .
		\end{align*}
where the third equality follows from Lemma \ref{sufficient}.
		
\noindent	``$\Rightarrow$'':  Suppose that $\sigma, \sigma'$ are such that $P_{\sigma,\cdot}=P_{\sigma',\cdot}$ and let $\pi=\sigma^{-1}\sigma'\in\Sym(S_M)$.  Then by Proposition \ref{minor}, 
	{	\[{\tiny
		\prod_{k\in S_M\setminus S_N} x_{\alpha_{\sigma(k)}\beta_{k}}^{(r_k)} 
		\left|  \begin{matrix}
		x_{\alpha_{\sigma(j_1)},\beta_{j_1}}^{(r_{j_1})} & \hdotsfor{1} & x_{\alpha_{\sigma(j_1)},\beta_{j_v}}^{(r_{j_1})} \\
		\vdots & \ddots & \vdots \\
		x_{\alpha_{\sigma(j_v)},\beta_{j_1}}^{(r_{j_v})} & \hdotsfor{1} & x_{\alpha_{\sigma(j_v)},\beta_{j_v}}^{(r_{j_v})}
		\end{matrix}
		\right| 
		  = \pm \prod_{k\in S_M\setminus S_N} x_{\alpha_{\sigma'(k)}\beta_{k}}^{(r_k)} \left|  \begin{matrix}
		x_{\alpha_{\sigma'(j_1)},\beta_{j_1}}^{(r_{j_1})} & \hdotsfor{1} & x_{\alpha_{\sigma'(j_1)},\beta_{j_v}}^{(r_{j_1})} \\
		\vdots & \ddots & \vdots \\
		x_{\alpha_{\sigma'(j_v)},\beta_{j_1}}^{(r_{j_v})} & \hdotsfor{1} & x_{\alpha_{\sigma'(j_v)},\beta_{j_v}}^{(r_{j_v})}
		\end{matrix}
		\right| . 
		}\] }
	\normalsize 
	
\noindent 
Taking the greatest common divisors of both sides, we get 

(a)  $\prod_{k\in S_M\setminus S_N}x_{\alpha_{\sigma(k)},\beta_{k}}^{(r_k)} 
= \prod_{k\in S_M\setminus S_N} x_{\alpha_{\sigma'(k)},\beta_{k}} ^{(r_k)}$; 

(b) the two determinants (which are pseudominors of $B$) are equal up to a sign.

It follows from (a) that $\alpha_{\sigma(k)}=\alpha_{\sigma'(k)}$ for $k\in S_M\setminus S_N$.  
Since $\alpha_i$ ($i\in S_M$) are all distinct, 
we have $\sigma(k)=\sigma'(k)$, $\pi(k)=k$, for all $k\in S_M\setminus S_N$.
It follows from (b) that for any $j$, the tuples 
$\big(\alpha_{\sigma(k)}\big)_{k\in I_j}$ and 
$\big(\alpha_{\sigma'(k)}\big)_{k\in I_j}$ are equal up to a permutation. 
Note that for $k\in S_M\cap S_N$, $\sigma(k)$ and $\sigma'(k)$ are also in $S_M\cap S_N$, so 
 $\alpha_i $ (for $i\in S_M\cap S_N$) are all distinct.
Thus for $k\in I_j$, $\pi(k)=\sigma^{-1}\sigma'(k)\in I_j$. 
Then $  \pi \in \prod\Sym(I_j)$ and $\overline{\sigma}=\overline{\sigma'} \text{ in } \overline{\Sym(S_M)}$.
	\end{proof}

\subsection{Violation}	
	\begin{defn}\label{violation}
	Use the notation in Definition \ref{def}. 
		Let $i,j,k$ be distinct indices such that $i\in S_M,j\in S_N,k\in S_M\cap S_N$. 
		If
		$\alpha_i \le \alpha_j<\alpha_k$, 
		$\beta_j \le \beta_i<\beta_k$,
		$(\alpha_i,\beta_i)\neq(\alpha_j,\beta_j)$, 
		and 
		$r_i=r_j=r_k$, 
		then the triple $(p_i,p_j,p_k)$ is called a \textit{violation} of $(M,N)$.  
		If furthermore the strict inequalities $\alpha_i < \alpha_j$ and $\beta_j < \beta_i$ hold, then the triple is called a \textit{strict violation} of $(M,N)$. 
	\end{defn}

	In light of the diagrams introduced in Example \ref{example1}, we may visualize a violation as the intersection of a horizontal ray and a vertical ray lying NW of an incidence in the same page. There are three possible cases as shown in Figure \ref{fig:violation}.
	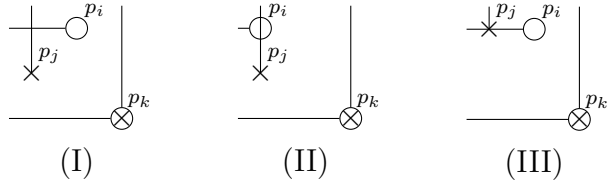
\begin{figure}[htbp]
		\centering
		$ 
		\begin{tikzpicture}[baseline={([yshift=-2em] current bounding box.north)},scale=0.6]
		\draw (1,0) circle (2.4mm);
		\draw (0,-1) node[]{$\times$};
		\draw (2,-2) node[]{$\times$};
		\draw (2,-2) circle (2.4mm);
		\draw (-.5,0) -- (0.75,0);
		\draw (0,.5) -- (0,-1);
		\draw (-.5,-2) -- (1.75,-2);
		\draw (2,.5)--(2,-1.75);
		\draw (1.4,0.4) node[]{\tiny{$p_i$}};
		\draw (0.4,0.4-1) node[]{\tiny{$p_j$}};
		\draw (2.4,0.4-2) node[]{\tiny{$p_k$}};
		\draw (1,-3) node{(I)};
		\end{tikzpicture}
		\hspace{1cm}
		\begin{tikzpicture}[baseline={([yshift=-2em] current bounding box.north)},scale=0.6]
		\draw (0,0) circle (2.4mm);
		\draw (0,-1) node[]{$\times$};
		\draw (2,-2) node[]{$\times$};
		\draw (2,-2) circle (2.4mm);
		\draw (-.5,0) -- (-0.25,0);
		\draw (0,.5) -- (0,-1);
		\draw (-.5,-2) -- (1.75,-2);
		\draw (2,.5)--(2,-1.75);
		\draw (0.4,0.4) node[]{\tiny{$p_i$}};
		\draw (0.4,0.4-1) node[]{\tiny{$p_j$}};
		\draw (2.4,0.4-2) node[]{\tiny{$p_k$}};
		\draw (1,-3) node{(II)};
		\end{tikzpicture}
		\hspace{1cm}
		\begin{tikzpicture}[baseline={([yshift=-2em] current bounding box.north)},scale=0.6]
		\draw (1,0) circle (2.4mm);
		\draw (0,0) node[]{$\times$};
		\draw (2,-2) node[]{$\times$};
		\draw (2,-2) circle (2.4mm);
		\draw (-.5,0) -- (0.75,0);
		\draw (0,.5) -- (0,0);
		\draw (-.5,-2) -- (1.75,-2);
		\draw (2,.5)--(2,-1.75);
		\draw (1.4,0.4) node[]{\tiny{$p_i$}};
		\draw (0.4,0.4) node[]{\tiny{$p_j$}};
		\draw (2.4,0.4-2) node[]{\tiny{$p_k$}};
		\draw (1,-3) node{(III)};
		\end{tikzpicture} $
		\caption{Three cases of violations, where (I) is a strict violation.
		}
		\label{fig:violation}
	\end{figure}
		Note that  $i\in S_M\setminus S_N$ and $j\in S_N\setminus S_M$. 
	
	Fix $(\sigma,\tau)\in G$ such that $L(\sigma,\tau)\neq L$. Let 
	$p_1>p_2>\hdots >p_l$, $q_1\geq q_2 \geq \hdots \geq q_l$
	be points corresponding to the variables of $L$ and $L(\sigma,\tau)$, respectively.  
	Define 
	\[ \jh =\min \{1\leq i\leq l :  p_i\neq q_i \}, \quad \kh =\min \{1\leq i\leq l:  (\sigma,\tau)p_i\neq p_i \} . \]
	It is easy to see that 
	\begin{equation}\label{atmostk}
		p_i \text{ is not a fixed point of } (\sigma,\tau) \quad\Longrightarrow \quad \kh\leq i 
		\quad \Longleftrightarrow \quad  p_{\kh}\ge p_i 
	\end{equation}

	\begin{lem}\label{notfixed}
		If $p_i$ is not a fixed point of $(\sigma,\tau)$, then so are  $p_{\sigma(i)}$ and $p_{\tau(i)}$.
	\end{lem}
	\begin{proof}
Assume $p_{\sigma(i)}$ is a fixed point. We shall show that $\sigma(i)=i$, thus  $p_i=p_{\sigma(i)}$ is a fixed point.
	This is obvious if $i\notin S_M$, so we assume $i\in S_M$. Then $\sigma(i)\in S_M$, and $(\sigma,\tau)p_{\sigma(i)}=p_{\sigma(i)}$ $\Rightarrow$ $(\alpha_{\sigma(i)},\beta_{\sigma(i)},r_{\sigma(i)})=p_{\sigma(i)}=(\sigma,\tau)^{-1}p_{\sigma(i)}=(\alpha_{i},\beta_{\tau^{-1}\sigma(i)},r_{\sigma(i)})$ $\Rightarrow$ $\alpha_{\sigma(i)}=\alpha_i$ 
	$\Rightarrow$ $\sigma(i)=i$. 
	\end{proof}

	\begin{lem}\label{largerimage}
		For any $i$, if $(\sigma,\tau)p_i > p_i$ and $(\sigma,\tau)p_i\geq p_{\kh}$, then the following holds:
		
		{\rm(i)} $i\in S_M\cap S_N$, $\sigma(i)\in S_M$, $\tau(i)\in S_N$;
		
		{\rm(ii)} $r_{\sigma(i)}=r_{\tau(i)}=r_i$;
		
		{\rm(iii)} $\alpha_{\sigma(i)}\leq \alpha_{\tau(i)}< \alpha_i$, $\beta_{\tau(i)}\leq \beta_{\sigma(i)}< \beta_i$.

\noindent As a consequence, the points $(\sigma,\tau)p_i$, $p_{\sigma(i)}$, $p_{\tau(i)}$ and $p_i$ all lie on the same $r_i$-th page, and there are four possible arrangements as shown in Figure \ref{fig:Remark:cases}. 
\begin{figure}[h]
		\begin{center}
			\begin{tikzpicture}[baseline={([yshift=-.5em] current bounding box.north)},scale=0.4]
			\draw[dashed] (-2,0) -- (4,0) ;
			\draw[dashed] (0,1.75) -- (0,-2)node[]{$\times$} ;
			\draw[dashed] (0,-2) -- (6,-2);
			\draw[dashed] (4,0) -- (4,-3.5);
			\filldraw (0,0) circle (1mm);
			\draw (4,0) circle (3.4mm);
			\draw (5.5,-3) circle (3.4mm) node[] {$\times$};
			\draw (4,0)node[anchor=south west]{$p_{\sigma(i)}$};
			\draw (0,-2)node[anchor=north east]{$p_{\tau(i)}$};
			\draw (0,0) node[anchor=south east] {$(\sigma,\tau)p_i$};
			\draw (5.5,-3) node[anchor=north west]{$p_i$};
			\node at (2,-5) {{\rm(I)} $\alpha_{\sigma(i)}<\alpha_{\tau(i)}, \beta_{\tau(i)}<\beta_{\sigma(i)}$}; 
			\end{tikzpicture}
			\hspace{.5cm}
			\begin{tikzpicture}[baseline={([yshift=-.5em] current bounding box.north)},scale=0.4]
			\draw[dashed] (-2,0) -- (0,0) ;
			\draw[dashed] (0,1.75) -- (0,-2)node[]{$\times$} -- (0,-3.5);
			\draw[dashed] (0,-2) -- (2,-2);
			\filldraw (0,0) circle (1mm);
			\draw (0,0) circle (3.4mm);
			\draw (1.5,-3) circle (3.4mm) node[] {$\times$};
			\draw (0,-2)node[anchor=north east]{$p_{\tau(i)}$};
			\draw (0,0) node[anchor=south west] {$(\sigma,\tau)p_i=p_{\sigma(i)}$};
			\draw (1.5,-3) node[anchor=north west]{$p_i$};
			\node at (0.5,-5) {{\rm(II)}  $\alpha_{\sigma(i)}<\alpha_{\tau(i)}, \beta_{\tau(i)}=\beta_{\sigma(i)}$}; 
			\end{tikzpicture} \\
			\vspace{.5cm}
			\begin{tikzpicture}[baseline={([yshift=-.5em] current bounding box.north)},scale=0.4]
			\draw[dashed] (-2,0) -- (5,0) ;
			\draw[dashed] (0,1.75) -- (0,0)node[]{$\times$} ;
			\draw[dashed] (3,0) -- (3,-1.5);
			\filldraw (0,0) circle (1mm);
			\draw (3,0) circle (3.4mm);
			\draw (4.5,-1) circle (3.4mm) node[] {$\times$};
			\draw (3,0)node[anchor=south west]{$p_{\sigma(i)}$};
			\draw (0,0) node[anchor=north east] {$(\sigma,\tau)p_i=p_{\tau(i)}$};
			\draw (4.5,-1) node[anchor=north west]{$p_i$};
			\node at (1.5,-3) {{\rm(III)}  $\alpha_{\sigma(i)}=\alpha_{\tau(i)}, \beta_{\tau(i)}<\beta_{\sigma(i)}$}; 
			\end{tikzpicture}
			\hspace{.5cm}
			\begin{tikzpicture}[baseline={([yshift=-.5em] current bounding box.north)},scale=0.4]
			\draw[dashed] (-2,0) -- (3,0) ;
			\draw[dashed] (0,1.75) -- (0,0)node[]{$\times$} -- (0,-1.5) ;
			\filldraw (0,0) circle (1mm);
			\draw (0,0) circle (3.4mm);
			\draw (1.5,-1) circle (3.4mm) node[] {$\times$};
			\draw (0,0) node[anchor=south west] {$(\sigma,\tau)p_i=p_{\sigma(i)}=p_{\tau(i)}$};
			\draw (1.5,-1) node[anchor=north west]{$p_i$};
			\node at (0.5,-3) {{\rm(IV)} $\alpha_{\sigma(i)}=\alpha_{\tau(i)}, \beta_{\tau(i)}=\beta_{\sigma(i)}$}; 
			\end{tikzpicture}
		\end{center}
		\caption{Figure of Lemma \ref{largerimage}}
		\label{fig:Remark:cases}
	\end{figure}
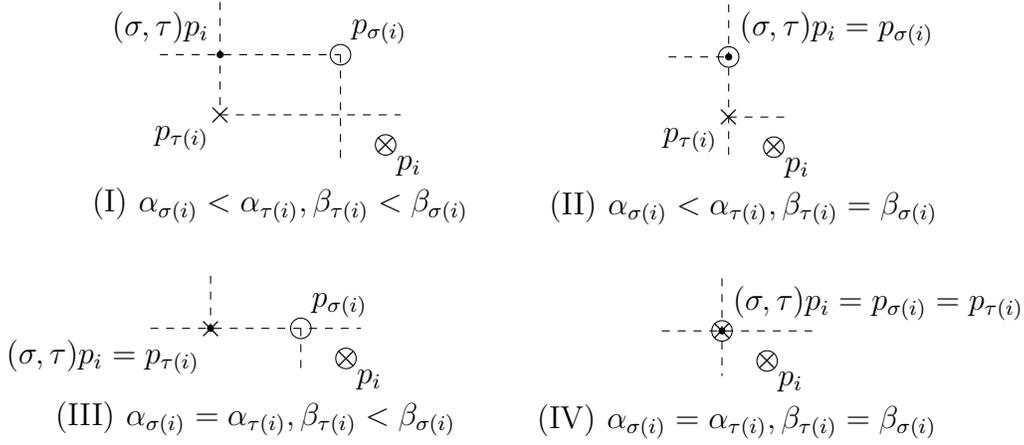
	\end{lem}
	
	\begin{proof}
	For convenience, we say that a point $(a_2,b_2)$ is weakly to the South of $(a_1,b_1)$ if $a_2\ge a_1$ and $b_2=b_1$; 
	$(a_2,b_2)$ is weakly to the SE of $(a_1,b_1)$ if $a_2\ge a_1$ and $b_2\ge b_1$; similar terminology will be used for other directions.

		Note that $p_{\tau(i)}=(\alpha_{\tau(i)},\beta_{\tau(i)},r_i)$ and $(\sigma,\tau)p_i=(\alpha_{\sigma(i)},\beta_{\tau(i)},r_i)$ are in the same column in the matrix $B$.  Since $p_i$ is not fixed, $p_{\tau(i)}$ is not fixed by Lemma \ref{notfixed}. Then $(\sigma,\tau)p_i\geq p_{\kh}\geq p_{\tau(i)}$ by \eqref{atmostk}.  
		Therefore, $p_{\tau(i)}$ is weakly to the South of $(\sigma,\tau)p_i$ in matrix $B$.  Similarly, $p_{\sigma(i)}$ is weakly to the East of $(\sigma,\tau)p_i$ in matrix $A$.  Thus,
		\begin{equation}\label{south}
			r_i<r_{\tau(i)} \textrm{ or } ( r_i=r_{\tau(i)}\text{ and } \alpha_{\sigma(i)}\leq \alpha_{\tau(i)} )
		\end{equation}
		\begin{equation}\label{east}
			r_i<r_{\sigma(i)} \textrm{ or } ( r_i=r_{\sigma(i)}  \text{ and }  \beta_{\tau(i)}\leq \beta_{\sigma(i)} ) 
		\end{equation}
		
		(i) Suppose that $i\notin S_M\cap S_N$. 
Without loss of generality, assume that $i\in S_M\setminus S_N$.  Then $\sigma(i)\in S_M$ and $\tau(i)=i$.  Observe that $\sigma(i)\neq i$ since
 $(\sigma,\tau)p_i \neq p_i$.  Therefore, \eqref{south} and \eqref{east} reduce to $\alpha_{\sigma(i)}\leq \alpha_i$ and ``$r_i<r_{\sigma(i)} \textrm{ or } ( r_i=r_{\sigma(i)}  \text{ and }  \beta_{i}\leq \beta_{\sigma(i)} )$''.  
Thus, $p_i$ and $p_{\sigma(i)}$ are distinct and not in NW-SE position in $A$,
which contradicts the fact that both $x_{p_i},x_{p_{\sigma(i)}}$ divide $\LT(M)$.  So $i\in S_M\cap S_N$, and thus $\sigma(i)\in S_M$ and $\tau(i)\in S_N$.  
		
		(ii) By symmetry we only need to show  $r_{\sigma(i)}=r_i$. Assume not, then by \eqref{east} we have $r_i<r_{\sigma(i)}$, thus $p_i>p_{\sigma(i)}$. 
It follows from (i) that both $i$ and $\sigma(i)$ are in $S_M$, so $p_i$ is to the NW of $p_{\sigma(i)}$ in $A$, thus $\alpha_i<\alpha_{\sigma(i)}$. 
Then  $\beta_i> \beta_{\tau(i)}$ because of the assumption $(\sigma,\tau)p_i>p_i$. 
Then $p_{\tau(i)}$ is to the NW of $p_i$ in $B$ because both $i$ and $\tau(i)$ are in $S_N$.
Comparing with \eqref{south},  we get
$r_{\tau(i)}=r_i$ 
and $\alpha_{\sigma(i)}\le \alpha_{\tau(i)}<\alpha_i$, but this contradicts the conclusion $\alpha_i<\alpha_{\sigma(i)}$.

		(iii) By (ii), \eqref{south} and \eqref{east} reduce to 
$\alpha_{\sigma(i)}\leq \alpha_{\tau(i)}$, $\beta_{\tau(i)}\leq \beta_{\sigma(i)}$; so $p_{\sigma(i)}$ is weakly to the NE of $p_{\tau(i)}$. 
It remains to prove
$\alpha_{\tau(i)}< \alpha_i$ and $\beta_{\sigma(i)}< \beta_i$. By symmetry we only prove the former. We consider two cases. 
Case 1: $\alpha_{\tau(i)}=\alpha_i$. Then $\tau(i)=i$,  $p_{\tau(i)}$ is weakly to the NE of $p_i=p_{\sigma(i)}$. But this is only possible when $\tau(i)=i$.  It then implies $(\sigma,\tau)p_i=p_i$, a contradiction. 
Case 2: $\alpha_{\tau(i)}>\alpha_i$. Then $\tau(i)>i$,  $p_{\tau(i)}$ is to the SE of $p_i$, and $\beta_i<\beta_{\tau(i)}\le \beta_{\sigma(i)}$. But then $(\sigma,\tau)p_i$ lies to the SE of $p_i$, which implies $(\sigma,\tau)p_i<p_i$, again a contradiction.
	\end{proof}
	
	\begin{prop}\label{noviolation}
Use the notation in Definition \ref{def}. 

{\rm(a)} $L(\sigma,\tau)=L$ if and only if $(\sigma,\tau)\in H$.
		
{\rm(b)} There exists a strict violation of $(M,N)$ if and only if there exists $(\sigma,\tau)\in G$ such that $L(\sigma,\tau) > L$. 
	\end{prop}
	\begin{proof} 
	
{\rm(a)}``$\Leftarrow$'' is a special case of Lemma \ref{sufficient}. For ``$\Rightarrow$'', 
we prove by induction on $m(M,N):=\min(\deg(M),\deg(N))$ for all $(\sigma,\tau)$. 
Fix $(\sigma,\tau)$ satisfying $L(\sigma,\tau)=L$, thus 
 $(\sigma,\tau)p_{\delta_i}=p_i$ ($1\le i\le l$) where $\delta_1,\dots,\delta_l$ is a rearrangement of $1,\dots,l$.		
In the base case $m(M,N)=0$, assume without loss of generality that $\deg(M)=0$, then $L(\sigma,\tau)=L(1,\tau)=L$ implies $\tau=1$, and $(\sigma,\tau)=(1,1)\in H$.  

For the inductive step,	we assume that $(\sigma,\tau)p_i\neq p_i$ ($1\le i\le l$), that is, $\delta_i\neq i$ for all $i$. Otherwise, assume $\delta_i=i$ for some $i$, then we can define $M'\in D_{u-1}(A)$ (resp.~$N'\in D_{v-1}(B)$)  to be obtained from $M$ (resp.~$N$) by removing the row and column of the corresponding submatrix that contains $x_{p_i}$. Applying the inductive hypothesis to $(M',N')$ with the same $(\sigma,\tau)$ , we get the conclusion that $(\sigma,\tau) = (\pi,\pi)$ where $\pi$ permutes the indices of points $p_j$ ($j\neq i$) while keeping their page numbers unchanged. Thus $(\sigma,\tau)\in H$. 

If $i\in S_M\cap S_N$, then $\delta_i\in S_M\cap S_N$, and both $\sigma$ and $\tau$ sends $\delta_i$ to $i$. In this case, define the transposition $\pi=(i \, \delta_i)$  and replace $(\sigma,\tau)$ by $(\sigma\pi,\tau\pi)$. Then $(\sigma,\tau)p_i=p_i$ is a fixed point of $(\sigma,\tau)$, a case that is already settled above. 

So we can assume $S_M\cap S_N=\emptyset$. It follows from the condition 
$(\sigma,\tau)p_{\delta_i}=p_i$ that:

-- if $i\in S_M$, then $\delta_i, \tau(\delta_i)\in S_N$, the three distinct points $p_i$, $p_{\delta_i}$, $p_{\tau(\delta_i)}$ are in the same page,  $p_i$ is in the same row with $p_{\delta_i}$, and in the same column with $p_{\tau(\delta_i)}$; 

-- if $i\in S_N$, then $\delta_i, \sigma(\delta_i)\in S_M$, the three distinct points $p_i$, $p_{\delta_i}$, $p_{\sigma(\delta_i)}$ are in the same page,  $p_i$ is in the same column with $p_{\delta_i}$, and in the same row with $p_{\sigma(\delta_i)}$. 

It follows that those $p_i$ in the same page must form a zig-zag chain as showing in Figure \ref{fig:zig-zag}. But then the endpoints of the zig-zag chain do not satisfy the condition, a contradiction. 
		\begin{figure}[h!]
			\centering
			\begin{tikzpicture}[baseline={([yshift=-2.5em] current bounding box.north)},scale=.6]
			\draw (-.75,0) -- (0,0) node[]{$\times$} -- (0,-1)node[]{$\Circle$} -- (1,-1)node[]{$\times$} -- (1,-2)node[]{$\Circle$} -- (1.75,-2) ;
			\draw (-1.2,0) node[]{$\hdots$};
			\draw (2.2,-2) node[]{$\hdots$};
			\end{tikzpicture}
			or
			\begin{tikzpicture}[baseline={([yshift=-2.5em] current bounding box.north)},scale=.6]
			\draw (-.75,0) -- (0,0) node[]{$\Circle$} -- (0,-1)node[]{$\times$} -- (1,-1)node[]{$\Circle$} -- (1,-2)node[]{$\times$} -- (1.75,-2) ;
			\draw (-1.2,0) node[]{$\hdots$};
			\draw (2.2,-2) node[]{$\hdots$};
			\end{tikzpicture}
			\caption{Zig-zag chain}
			\label{fig:zig-zag}
		\end{figure}
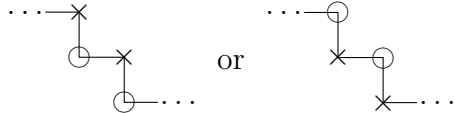

{\rm(b)}	``$\Rightarrow$'': Given a strict violation $(p_i,p_j,p_k)$, the transpositions $\sigma=(i\,k)$ and $\tau=(j\,k)$ satisfy  $L(\sigma,\tau) > L$. 
	
``$\Leftarrow$'':	
	 Suppose that $L(\sigma,\tau) > L$ and there is no strict violation of $(M,N)$. Similar to the proof of (a), we can assume that $(\sigma,\tau)p_i\neq p_i$ for all $1\le i\le l$ (which implies $\delta_i\neq i$ for $1\le i< \jh$), and $i\notin S_M\cap S_N$ for all $i<\jh$. 
By symmetry, we assume that $1\in S_M$. 

First, we consider the special case $\jh=1$. 
The condition $L(\sigma,\tau) > L$ implies that $(\sigma,\tau)p_{\delta_1}=q_1>p_1>p_{\delta_1}$. So $q_1,p_1,p_{\delta_1}$ are in the same page. 

If $\delta_1\in S_M\setminus S_N$, then $p_{\delta_1}$ is to the SE of $p_1$ thus $p_{\sigma(\delta_1)}$ and $q_1$ are on the right side of column containing $p_1$.  So $q_1$ is to the NE of $p_1$. It implies that $p_{\sigma(\delta_1)}$ is above the row containing $p_1$. Then $p_{\sigma(\delta_1)}>p_1$ since both $\sigma(\delta_1)$ and $p_1$ are in $S_M$. This contradicts the fact that $p_1$ is largest. 

If $\delta_1\in S_N\setminus S_M$, we consider two subcases.
If $p_{\tau(\delta_1)}$ is weakly to the NW of $p_{\delta_1}$, then it is weakly above $q_1$ while in the same column, so  $p_{\tau(\delta_1)}\ge q_1>p_1$, a contradiction. 
If $p_{\tau(\delta_1)}$ is to the SE of $p_{\delta_1}$, then $p_{\delta_1}$ is in the same row of $q_1$ and on its left, thus  $p_{\delta_1}\ge q_1>p_1$, a contradiction.

If $\delta_1\in S_M\cap S_N$, we consider two subcases. 
If $q_1$ is above the row containing $p_1$, then $\sigma(\delta_1)$ is in $S_M$ and $p_{\sigma(\delta_1)}$ is above the row containing $p_1$, thus $p_{\sigma(\delta_1)}>p_1$, a contradiction. 
If $q_1$ is to the left of the column containing $p_1$, then $\tau(\delta_1)$ is in $S_N$ and $p_{\tau(\delta_1)}$ is on the left side of the column containing $p_1$,  and $p_1$ and $p_{\tau(\delta_1)}$ are both to the NW of $p_{\delta_1}$. Thus $(p_1, p_{\tau(\delta_1)},p_{\delta_1})$ is a strict violation of $(M,N)$, contradicting the assumption. 

This completes the proof of the case $\jh=1$. Next, we consider the case $\jh>1$. 

First note that $p_1$ and $p_{\delta_1}$ are obviously in the same page. Moreover, 
$p_1$ lies to the NW of $p_{\delta_1}$,  
$\delta_1\in S_M\cap S_N$,
$\sigma(\delta_1)=1$.
 Indeed, to show $\delta_1\in S_M\cap S_N$:
if $\delta_1\in S_M\setminus S_N$, then  $p_1=q_1=(\sigma,\tau)p_{\delta_1}$ lies in the same column with $p_{\delta_1}$; but $1$ and $\delta_1$ are in $S_M$, so $p_1$ is to the NW of $p_{\delta_1}$,  a contradiction. 
If $\delta_1\in S_N\setminus S_M$, then $p_1$ is in the same row with and to the left of $p_{\delta_1}$, so $p_{\tau(\delta_1)}$ is in the same column with $p_1$ and above $p_1$, which implies  $p_{\tau(\delta_1)}>p_1$, a contradiction. 
Thus $\delta_1\in S_M\cap S_N$. It follows that $p_1$ lies to the NW of $p_{\delta_1}$. 
The identity $\sigma(\delta_1)=1$ follows from the fact that $1\in S_M$.

We show that $\tau(\delta_1)=2$ (thus $p_1$ and $p_2$ are in the same column and $p_2$ is below $p_1$). 
Assume the contrary that $\tau(\delta_1)\neq 2$. 
The condition $(\sigma,\tau)p_{\delta_1}=q_1=p_1$ implies that $p_{\sigma(\delta_1)}$ is at the same row as $p_1$, while 
$p_{\tau(\delta_1)}$ is at the same column as $p_1$. 
If $2\in S_N$, then $(p_1,p_{2},p_{\delta_1})$ is a strict violation, contradiction our assumption; 
if $2\in S_M$, then
$(p_{2},p_{\tau(\delta_1)},p_{\delta_1})$ is a strict violation, again a contradiction. 

If $\jh=2$, then $q_2>p_2$. So  $q_2$ is either above the row containing $p_2$, or to the left of the column containing $p_2$.
In the former case, $q_2$ shares the row with $p_{\sigma(\delta_2)}$, so $p_{\sigma(\delta_2)}=p_1$, $\sigma(\delta_2)=1$. However, $\sigma(\delta_1)=1$ and $\delta_2\neq \delta_1$, contradicting to the fact that $\sigma$ is bijective. 
In the latter case, $p_{\tau(\delta_2)}$ is to the left of the column containing $p_1\in S_M$ and $p_2\in S_N$, so $\tau(\delta_2)\neq1,2$ and  $p_{\tau(\delta_2)}>p_2$, which is impossible. 

If $\jh\ge 3$, then $q_2=p_2$, $\tau(\delta_2)=2$, but $\tau(\delta_1)=2$, contradicting the bijectivity of $\tau$. 
\end{proof}

	\subsection{$P(M,N)$}
		
	\begin{defn} \label{main}
	(i) For each $M,N\in D_u(A)\cup D_v(B)$, define an expression
\begin{equation}\label{eq:PMN}
 P(M,N) =\sum_{\substack{\overline{\sigma}\in \overline{\Sym(S_M)} \\ \overline{\sigma}\neq \overline{1}}} P_{\overline{\sigma},\cdot}-\sum_{\substack{\overline{\tau}\in \overline{\Sym(S_N)} \\ \overline{\tau}\neq \overline{1}}} P_{\cdot,\overline{\tau}}
 \end{equation}
		where each trivial pseudominor is replaced by 0.
		
	(ii) We say that $P(M,N)$ has \emph{sufficiently small leading terms} 
if \emph{all} the leading terms of $P_{\overline{\sigma},\cdot}$ (for $\overline{\sigma}\neq \overline{1}$) and of $P_{\cdot,\overline{\tau}}$ (for $\overline{\tau}\neq \overline{1}$) in the right side of 
\eqref{eq:PMN} are less than $L(=P_{1,1})$. Otherwise we say that $P(M,N)$ does not have sufficiently small leading terms. 
	\end{defn}   
	
\begin{prop} \label{P=S}
		For each $M,N\in D_u(A)\cup D_v(B)$, $S(M,N)=P(M,N)$. 
\end{prop}

	\begin{proof}
We assume $M\neq N$ since otherwise the statement is trivial.  
Let $h=|H|\;(=|\prod\Sym(I_j)|)$ and let $ \{1, \sigma_1,\hdots,\sigma_{h'}\}$ and $\{1,\tau_1,\hdots,\tau_{h''}\} $ be complete sets of coset representatives for $\Sym(S_M)/\prod\Sym(I_j) $ and $\Sym(S_N)/\prod\Sym(I_j)$, respectively.  Then, the cardinality of each coset in $\Sym(S_M)/\prod\Sym(I_j) $ and $\Sym(S_N)/\prod\Sym(I_j)$ is $h$, and by Proposition \ref{equalP}, we can write \eqref{doublesum} as 
		$hP_{1,\cdot}+hP_{\sigma_1,\cdot}+\hdots +hP_{\sigma_{h'},\cdot}=hP_{\cdot,1}+ hP_{\cdot,\tau_1}+\hdots+hP_{\cdot,\tau_{h''}}$. 
Dividing by $h$ (to avoid dividing by 0 when $K$ has positive characteristic, we can prove the equality over $\mathbb{Z}$ instead of $K$) and 
		using Proposition \ref{minor}, we get
		\[ \dfrac{L}{\LT(N)}N+\sum_{\substack{\overline{\sigma}\in \overline{\Sym(S_M)} \\ \overline{\sigma}\neq \overline{1}}} P_{\sigma,\cdot} = \dfrac{L}{\LT(M)}M+\sum_{\substack{\overline{\tau}\in \overline{\Sym(S_N)} \\ \overline{\tau}\neq \overline{1}}} P_{\cdot,\tau}  \]
		Rearranging the terms gives the desired equality. 
	\end{proof}

	\begin{lem}\label{noL}
		The polynomial $P(M,N)$ contains only monomials of the form $\pm L(\sigma,\tau)$, none of which are equal to $\pm L$.  
	\end{lem}
	\begin{proof}
	If $L(\sigma,\tau)=L$, then $(\sigma,\tau)=(\pi,\pi)$ with $\pi\in \prod\Sym(I_j)$ by Proposition \ref{noviolation} (a). Then 
$\bar{\sigma}=\bar{1}\in\Sym(S_M)/\prod\Sym(I_j)$, 
$\bar{\tau}=\bar{1}\in\Sym(S_N)/\prod\Sym(I_j)$,
so $L(\sigma,\tau)$ does not appear in either sum of the right side of \eqref{eq:PMN}, therefore is not contained in $P(M,N)$.
	\end{proof}

	\begin{prop}\label{lessthanL}
		For any $M\in D_u(A),N\in D_v(B)$, if there is no strict violation of $(M,N)$, then $P(M,N)$ has sufficiently small leading terms.
	\end{prop}
	
\begin{proof}
Assume there is no strict violation of $(M,N)$.  
By Definition \ref{main} and Lemma \ref{noL}, $P(M,N)$ consists of a sum of monomials of the form $L(\sigma,\tau)\neq L$.  Since there exists no strict violation of $(M,N)$, each $L(\sigma,\tau)$ is less than $L$ by Proposition \ref{noviolation}.  In particular, the leading terms of $P_{\overline{\sigma},\cdot}$ (for $\overline{\sigma}\neq \overline{1}$) and of $P_{\cdot,\overline{\tau}}$ (for $\overline{\tau}\neq \overline{1}$) are all less than $L$, so  $P(M,N)$ has sufficiently small leading terms.
	\end{proof}

\subsection{Distance}\label{Distance}
Use the notation in Definition \ref{def}. 

\begin{defn}\label{distance}
	Let $\mathcal{D}=\Big( \bigcup_u D_u(A) \Big) \cup \Big( \bigcup_v D_v(B) \Big)$ and for any $P\in \mathcal{D}$ define $V_P=\{ x_{ij}^{(k)}:x_{ij}^{(k)}\textrm{ divides }\LM(P) \} $.  Define the \textit{distance} $d(M,N)$ for  $M,N\in \mathcal{D}$ to be
	$ d(M,N)=|(V_M\setminus V_N)\cup (V_N\setminus V_M)|=|V_M|+|V_N|-2|V_M \cap V_N|$. 
\end{defn}    

\begin{defn}\label{adjacent}
For $M\in D_{u}(A),N\in D_{v}(B)$, let $L=\LCM(\LM(M),\LM(N))$. 
	We say that $M$ and $N$ are \textit{$(A,B)$-adjacent} if $M\neq N$ and both of the following are true:
	
\noindent	(i) There does not exist a minor $P\in D_u^L(A):=\{ P\in D_u(A):\LM(P) \textrm{  divides } L\}$ such that $d(P,N)<d(M,N)$;
	
\noindent	(ii) There does not exist a minor $Q\in D_v^L(B):=\{ P\in D_v(B):\LM(P) \textrm{  divides }  L\}$ such that $d(M,Q)<d(M,N)$.
	
For $M,N\in D_u(A)$, we say that $M$ and $N$ are \textit{$A$-adjacent} if $M\neq N$ and there does not exist $P\in D_u(A)$ such that both $d(M,P)<d(M,N)$ and $d(P,N)<d(M,N)$. 
\end{defn}

\subsection{Single Determinantal Ideal}\label{single}
In this subsection we  study the single determinantal ideals using a method inspired by \cite{Narasimhan} and \cite{Caniglia} which will be generalized to the bipartite determinantal ideals.
Assume that $r=1$, $u=v\le\min(m,n)$, denote $x_{ij}^{(1)}$ as $x_{ij}$, $L= x_{\alpha_1\beta_1}\hdots x_{\alpha_l\beta_l}$, the points $p_i=(\alpha_i,\beta_i)$ satisfying $p_1>p_2>\cdots>p_l$.

\begin{defn} \label{defect}
We say that $(p_j,p_k,p_r,p_s,p_t)$ is a \textit{defect} of type I (resp.~type II) if there exist distinct indices 
$j,k,r,s,t$ such that ``$j,s\in S_M\setminus S_N$ and $k,r\in S_N\setminus S_M$'' (resp.~``$j,s\in S_N\setminus S_M$ and $k,r\in S_M\setminus S_N$''), 
$t\in S_M\cap S_N$,  $\alpha_j \le \alpha_k<\alpha_r \le \alpha_s<\alpha_t$, $\beta_k \le \beta_j<\beta_s \le \beta_r<\beta_t$. 
We say that the pair $(M,N)$ is \textit{defective} if a defect of either type exists. 
	
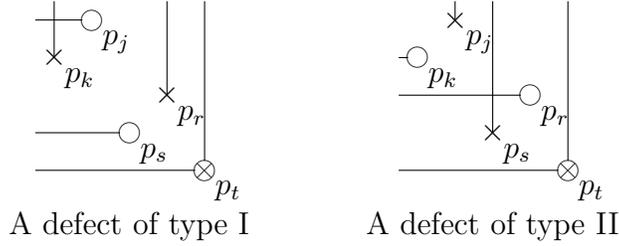
\begin{center}
		\begin{tikzpicture}[baseline={([yshift=-.5em] current bounding box.north)},scale=0.5]
		\draw (-.5,0) -- (1-.25,0);
		\draw (1,0) node[]{$\Circle$};
		\draw (1,0) node[anchor=north west]{$p_j$};
		\draw (0,.5) -- (0,-1)node[]{$\times$};
		\draw (0,-1) node[anchor=north west]{$p_k$};
		\draw (3,.5)--(3,-2)node[]{$\times$};
		\draw (-.5,-3)--(2-.25,-3);
		\draw (2,-3) node[]{$\Circle$}; 
		\draw (-.5,-4)--(4-.25,-4);
		\draw (4,-4+.25)--(4,.5);
		\draw (4,-4)node[]{$\otimes$};
		\draw (2,-5.5) node[]{A defect of type I};
		\draw (2,-3) node[anchor=north west]{$p_s$};
		\draw (3,-2) node[anchor=north west]{$p_r$};
		\draw (4,-4) node[anchor=north west]{$p_t$};
		\end{tikzpicture}
		\hspace{1cm}
		\begin{tikzpicture}[baseline={([yshift=-.5em] current bounding box.north)},scale=0.5]
		\draw (-.5,-1) -- (0-.25,-1);
		\draw (0,-1) node[]{$\Circle$};
		\draw (1,.5) -- (1,0)node[]{$\times$};
		\draw (0,-1) node[anchor=north west]{$p_k$};
		\draw (1,0) node[anchor=north west]{$p_j$};
		\draw (2,.5)--(2,-3)node[]{$\times$};
		\draw (-.5,-2)--(3-.25,-2);
		\draw (3,-2) node[]{$\Circle$}; 
		\draw (-.5,-4)--(4-.25,-4);
		\draw (4,-4+.25)--(4,.5);
		\draw (4,-4)node[]{$\otimes$};
		\draw (2,-5.5) node[]{A defect of type II};
		\draw (2,-3) node[anchor=north west]{$p_s$};
		\draw (3,-2) node[anchor=north west]{$p_r$};
		\draw (4,-4) node[anchor=north west]{$p_t$};
		\end{tikzpicture}
	\end{center}

	A defect $(p_j,p_k,p_r,p_s,p_t)$ of type I (resp. of type II) is called \textit{maximal} if it satisfies the following two conditions:		
	(i)  If $j'\leq j$,  $k'\leq k$, $(p_{j'},p_{k'},p_t)$ is a violation of $(M,N)$ (resp.~of $(N,M)$), then $j'=j$ and $k'=k$; 
	(ii) If $j<s'\leq s$,  $k<r'\leq r$, and $(p_j,p_k,p_{s'},p_{r'},p_t)$ is a defect of type I (resp.~type II), then  $s'=s$ and $r'=r$.
\end{defn}

Intuitively, a maximal defect is such that both pairs $(p_j,p_ k)$ and $(p_r,p_s)$ are located as far NW as possible. Obviously, a maximal defect of type I (resp.~type II) exists if a defect of type I (resp.~type II) exists.

\begin{lem}\label{combination} For $M,N\in D_u(A)$, the pair $(M,N)$ is defective (of either type) if and only if there exist both a violation of $(M,N)$ and a violation of $(N,M)$. Consequently, if $(M,N)$ is not defective, then either $P(M,N)$ or $P(N,M)$ has sufficiently small leading terms.
\end{lem}

\begin{proof}
The first statement is obvious, the second one uses Proposition \ref{lessthanL}. 
\end{proof}

\begin{prop}\label{notdefective}
	If $M,N\in D_u(A)$ are $A$-adjacent, then $(M,N)$ is not defective.
\end{prop}
\begin{proof}
	We will prove the contrapositive.  Suppose that $(M,N)$ is defective.  Without loss of generality we assume that there exists a defect of type I.  Inspired by the proof in  \cite{Caniglia}, we will construct a new minor between $M$ and $N$ by replacing a section of the ordered sequence of points defining the leading term of one of the minors (the recipient) with a section of equal size from the leading term of the other (the donor).  
		
Arrange the points in $\{(\alpha_i,\beta_i)\, :\, i\in S_M\}$ in decreasing order and denote them as 
 $m_1>m_2>\hdots >m_u$. Similarly, denote points in $\{(\alpha_i,\beta_i)\, : \,  i\in S_N\}$  as $n_1>n_2>\hdots>n_u$. Denote $m_i=(a_i,b_i),n_i=(a'_i,b'_i)$ for $i=1\hdots u$.  
 Let $(m_j,n_k,n_r,m_s,m_t=n_{t'}) $ be a maximal defect of type I (note that the indices are chosen in a different way than those in Definition \ref{defect}).  Define 
$ w_1=\min \{i:i>j,\ m_i \text{ is an incidence}\}$,
$w_2=\min \{i:i>k,\ n_i \text{ is an incidence}\}$. 
 Note that $w_1\leq t$ and $w_2\leq t'$.  Define $y_1=\min \{s,w_1\} - j$ and $y_2=\min \{ r,w_2\} -k$. Note that $y_1,y_2\in\mathbb{Z}_{>0}$. 
 Define the minor $P$ such that 
\begin{equation}\label{eq:minorP}
\LT(P)=
\begin{cases}
&	x_{m_1}\cdots x_{m_{j-1}}x_{n_{k}}\cdots x_{n_{k+y_1-1}} x_{m_{j+y_1}}\cdots x_{m_u},  \text{ if $y_1\leq y_2$;} \\
&	x_{n_1}\cdots x_{n_{k-1}}x_{m_{j}}\cdots x_{m_{j+y_2-1}}x_{n_{k+y_2}}\cdots x_{n_u}, \text{ if $y_2< y_1$.}
\end{cases}
\end{equation}
Note that when $y_1\leq y_2$, $LT(P)$ is obtained from $LT(M)$ by replacing $y_1$ variables in $LT(M)$ with $y_1$ variables in $LT(N)$, so we say that $M$ is the recipient and $N$ is the donor; when $y_2<y_1$, $LT(P)$ is obtained from $LT(N)$ by replacing $y_2$ variables in $LT(N)$ with $y_2$ variables in $LT(M)$, so we say that $N$ is the recipient and $M$ is the donor.
	Figure \ref{fig:transplant} visually demonstrates two nearly identical examples of how $\LT(P)$ is selected.  
	\begin{figure}[h]
		\centering
		\begin{tikzpicture}[scale=0.35]
		\draw (0,0) node[]{$\times$};
		\draw (1,-1) circle (3.3mm);
		\draw (3,-2) circle (3.3mm);
		\draw (4,-3) circle (3.3mm);
		\draw (2,-4) node[]{$\times$};
		\draw (6,-6) node[]{$\times$};
		\draw (7,-7) node[]{$\times$};
		\draw (8,-9) circle (3.3mm);
		\draw (9,-8) node[]{$\times$};
		\draw (10,-10) circle (3.3mm) node[] {$\times$};
		\draw (11,-11) circle (3.3mm);
		\draw (2,0) -- (2,-4);
		\draw (-1,-2) --(2.68,-2);
		\fill[gray,fill opacity=0.3] (2,-1) rectangle (3,-2);
		\fill[gray,fill opacity=0.3] (0,-2) rectangle (2,-4);
		\draw (3,-8) -- (9,-8);
		\draw (8,-6) --(8,-8.68);
		\fill[gray,fill opacity=0.3] (4,-8) rectangle (8,-9);
		\fill[gray,fill opacity=0.3] (8,-7) rectangle (9,-8);
		\draw (3,-2) node[anchor=south west]{$m_j$};
		\draw (2,-4) node[anchor=north east]{$n_k$};
		\draw (9,-8) node[anchor=south west]{$n_r$};
		\draw (8,-9) node[anchor=north east]{$m_s$};
		\draw (10,-10) node[anchor=south west]{$m_t=n_{t'}$};
		\draw (6,-12) node []{$y_1 = s-j=2 $ };
		\draw (6,-13) node []{$y_2 =r-k=3$ };
		\draw (6,-14) node []{$N$ is donor };	
		\draw (1-.35,-1-.35) rectangle (1+.35,-1+.35);
		\draw (2-.35,-4-.35) rectangle (2+.35,-4+.35);
		\draw (6-.35,-6-.35) rectangle (6+.35,-6+.35);
		\draw (8-.35,-9-.35) rectangle (8+.35,-9+.35);
		\draw (10-.35,-10-.35) rectangle (10+.35,-10+.35);
		\draw (11-.35,-11-.35) rectangle (11+.35,-11+.35);
		\end{tikzpicture}
		\begin{tikzpicture}[scale=0.35]
		\draw (0,0) node[]{$\times$};
		\draw (1,-1) circle (3.3mm);
		\draw (3,-2) circle (3.3mm);
		\draw (4,-3) circle (3.3mm);
		\draw (2,-4) node[]{$\times$};
		\draw (5,-5) circle (3.3mm) node[] {$\times$};
		\draw (6,-6) node[]{$\times$};
		\draw (7,-7) node[]{$\times$};
		\draw (8,-9) circle (3.3mm);
		\draw (9,-8) node[]{$\times$};
		\draw (10,-10) circle (3.3mm) node[] {$\times$};
		\draw (11,-11) circle (3.3mm);
		\draw (2,0) -- (2,-4);
		\draw (-1,-2) --(2.68,-2);
		\fill[gray,fill opacity=0.3] (2,-1) rectangle (3,-2);
		\fill[gray,fill opacity=0.3] (0,-2) rectangle (2,-4);
		\draw (4,-8) -- (9,-8);
		\draw (8,-6) --(8,-8.68);
		\fill[gray,fill opacity=0.3] (5,-8) rectangle (8,-9);
		\fill[gray,fill opacity=0.3] (8,-7) rectangle (9,-8);
		\draw (3,-2) node[anchor=south west]{$m_j$};
		\draw (2,-4) node[anchor=north east]{$n_k$};
		\draw (9,-8) node[anchor=south west]{$n_r$};
		\draw (8,-9) node[anchor=north east]{$m_s$};
		\draw (5,-5) node[anchor=south west]{$m_{w_1}=n_{w_2}$};
		\draw (10,-10) node[anchor=south west]{$m_t=n_{t'}$};
		\draw (6,-12) node []{$y_1 = w_1-j=2 $ };
		\draw (6,-13) node []{$y_2 =w_2-k=1$ };
		\draw (6,-14) node []{$M$ is donor };
		\draw (-.35,-.35) rectangle (.35,.35);
		\draw (3-.35,-2-.35) rectangle (3+.35,-2+.35);
		\draw (5-.35,-5-.35) rectangle (5+.35,-5+.35);
		\draw (6-.35,-6-.35) rectangle (6+.35,-6+.35);
		\draw (7-.35,-7-.35) rectangle (7+.35,-7+.35);
		\draw (9-.35,-8-.35) rectangle (9+.35,-8+.35);
		\draw (10-.35,-10-.35) rectangle (10+.35,-10+.35);
		\end{tikzpicture}
		\caption{Two examples of transplant procedure, where $\Circle=m_i$, $\times=n_i$, $\square = $ point chosen for $\LT(P)$. The diagram on the right contains one additional incidence in the middle of the picture, causing a different choice for $\LT(P)$.  The shaded regions do not contain any of those points by maximality.}
		\label{fig:transplant}
	\end{figure}
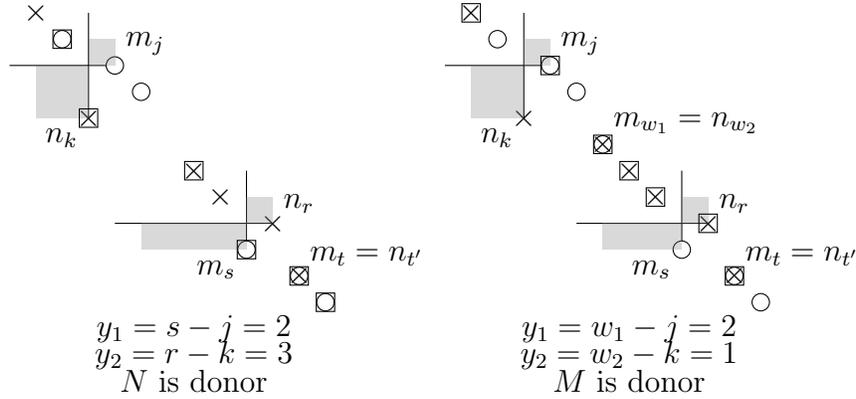
	
It is clear that $\LT(P)|L$ and $\deg(P)=u$.  By Definition \ref{defect}, none of $m_j,n_k,n_r$ or $m_s$ are incidences.  
By \eqref{eq:minorP}, 
 $y_1\leq s-j$ implies that $m_s$ is a point of $\LT(P)$ if $y_1\leq y_2$,   Similarly, $y_2\leq r -k$ implies that $n_r$ is a point of $\LT(P)$ if $y_2<y_1$.  Therefore, we will either choose both $n_k$ and $m_s$ if $M$ is the recipient, or both $m_j$ and $n_r$ if $N$ is the recipient.  
So $P\neq M,N$.
	
	In the rest of the proof we assume, without loss of generality, that $y_1\leq y_2$ (so $M$ is recipient and $N$ is donor).  To verify the variables in $\LT(P)$ are in NW-SE position, we need only check that $m_{j-1}$ and $n_k$ are in NW-SE position, $n_{k+y_1-1}$ and $m_{j+y_1}$ are also in NW-SE position.  That is, we must verify the strict inequalities. 
	\begin{align}
	&a_{j-1}<a'_k, \quad b_{j-1}<b'_k \label{ineq12}\\
	&a'_{k+y_1-1}<a_{j+y_1}, \quad b'_{k+y_1-1}< b_{j+y_1} \label{ineq34}
	\end{align}
	
	To verify \eqref{ineq12}, note that $m_{j-1}$ is NW of $m_j$, so $a_{j-1}<a_j$ and $b_{j-1}<b_j$.  By definition of defect, $a_j\leq a'_k$ and $b'_k\leq b_j$.  So $a_{j-1}<a'_k$.  If $b_{j-1}\ge b'_k$, then $(m_{j-1},n_k,m_t)$ is a violation of $(M,N)$, which violates Definition \ref{defect} condition (i) of maximality, so $b_{j-1}<b'_k$.
	
To verify \eqref{ineq34}, note that $s\neq w_1$ because $m_{w_1}$ is an incidence but $m_s$ is not. So we break into two cases.  
Case 1:  $s<w_1$. Then $s-j=y_1\leq y_2\le r-k$,  \eqref{ineq34} becomes  $a'_{k+s-j-1}<a_{s}$, $b'_{k+s-j-1}< b_{s}$.  Now, since $k+s-j-1\leq k+(r-k)-1 = r-1$, $n_{k+s-j-1}$ must be NW of $n_r$, so $a'_{k+s-j-1}<a'_r$ and $b'_{k+s-j-1}<b'_r$.  Also, by the definition of type I defect, we have $a'_r\leq a_s$ and $b_s\leq b'_r$, so $a'_{k+s-j-1}<a'_r\le a_s$ which proves the first equation of  \eqref{ineq34}. 
To prove the second equation of \eqref{ineq34}, we assume the contrary that $b_s\leq b'_{k+s-j-1}$. 
Note that $k<k+s-j-1$ (if $k=k+s-j-1$, then $n_k$ lies weakly SW of $m_j$ and weakly NE of $m_s$, forcing $m_j=m_s$, a contradiction).
Then $(m_j,n_k,n_{k+s-j-1},m_s,m_t)$ is a defect of type I, which together with inequalities $k<k+s-j-1<r$ contradict the fact the assumption that $(m_j,n_k,n_r,m_s,m_t)$ is a maximal defect of type I.
%
%
%
Case 2:  $s>w_1$.  So, $j<w_1<s$, $k<w_2<r$, $m_{w_1}=n_{w_2}$ is an incidence, $y_1=w_1-j$, and $y_2=w_2-k$.  Then \eqref{ineq34} amounts to showing that $n_{k+w_1-j-1}$ is NW of $m_{w_1}=n_{w_2}$, which follows from $k+w_1-j-1<k+y_1\le k+y_2=w_2$.    
	
Finally, we shall prove $d(M,P)+d(P,N)=d(M,N)$ and $d(M,P)>0,d(P,N)>0$, which implies that $M,N$ are not $A$-adjacent. Indeed, let $Z=V_M\cap V_N$ be the set of all incidences and let $z=|Z|$.  
Note that $x_{n_k},x_{n_{k+1}},\dots,x_{n_{k+y_1-1}}\in V_N\setminus V_M$. 
	Then	
$d(M,P) =  2y_1>0$, 
$d(P,N) = 2u-2(z+y_1)$, 
so $d(M,N)=d(M,P)+d(P,N)=2u-2z=d(M,N)$. Moreover, 
$k+y_1-1\le k+y_2-1\le k+(r-k)-1<r 
\Rightarrow n_r\in V_N\setminus V_P
\Rightarrow P\neq N
\Rightarrow d(P,N)>0$.
\end{proof}

\begin{prop} \label{thereexistsachain}
	For any $M,N\in D_u(A)$ and $L=\LCM(\LM(M),\LM(N))$, there exists a chain 
	$( M=M_0, M_1,\hdots,M_{k-1},M_k=N )$ 
	of minors in $D_u^L(A)$
	such that for each $j=1,\hdots,k$, the S-pair $S(M_{j-1},M_j)$ may be expressed as a finite sum $\sum a_iP_i$ where $a_i$ are monomials, $P_i\in D_u(A)$ and 
 $\LM(a_iP_i)<\LCM(\LM(M_{j-1}),\LM(M_j))$ for all $i$.
\end{prop} 

\begin{proof}  The proposition is trivial if $M=N$, so we assume $M\neq N$. Start with the chain $(M_0=M, M_1=N)$. If there are $M_{i-1},M_i$ which are distinct and not $A$-adjacent, then there exists $P\in D_u^L(A)$ such that  both $d(M_{i-1},P)$ and $d(P,M_i)$ are less than $d(M_{i-1},M_i)$, then we insert $P$ between $M_{i-1}$ and $M_i$, that is, re-index $M_j$ as $M_{j+1}$ for all $j\ge i$, and define $M_i=P$. 
 After finitely many insertions we obtain a chain $( M=M_0, M_1,\hdots,M_{k-1},M_k=N )$ such that $M_{j-1}$ and $M_j$ are $A$-adjacent for every $1\le j\le k$.  For every such $j$, $M_{j-1}$ and $M_j$ are not defective by Proposition \ref{notdefective},  thus by Lemma \ref{combination}, at least one of $P(M_{j-1},M_j)$ or $P(M_j,M_{j-1})$ may be expressed as a finite sum $\sum a_iP_i$ where $a_i$ are monomials, $P_i\in D_u(A)$ and $\LM(a_iP_i)<\LCM(\LM(M_{j-1}),\LM(M_j))$ for all $i$.  The claim follows, since $P(M_{j-1},M_j)=-P(M_j,M_{j-1})=S(M_{j-1},M_j)$.
\end{proof}

\subsection{Bipartite Determinantal Ideal}
The proof of the previous section can be easily adapted to the case of bipartite determinantal ideals.  	We use the notation in Definition \ref{def}. 

\begin{prop} \label{noviolation'}

	If $M,N$ are $(A,B)$-adjacent, then there is no violation of $(M,N)$.  
\end{prop}

\begin{proof}
We prove the contrapositive.  Suppose there exists a violation of $(M,N)$ (so $M\neq N$).  
%
	Let $\LM(M)$ and $\LM(N)$ be given by the points $m_1>m_2>\hdots>m_u$ and $n_1>n_2>\hdots>n_v$, respectively, where $m_i=(a_i,b_i,s_i)$ and $n_i=(a'_i,b'_i,t_i)$.   Let $(m_j,n_k,m_{w_1}=n_{w_2})$ be a violation that is maximal in the following sense:

\noindent (a)  If $j' \leq j$ and $k'\leq k$, then $(m_{j'},n_{k'},m_{w_1})$ is a violation only when $j' = j$ and $k' = k$, and 

\noindent (b) $w_1=\min \{i:i>j, \  m_i \text{ is an incidence} \}$ and $w_2=\min \{ i:i>k,\ n_i \text{ is an incidence} \} $. 

Without loss of generality, assume $w_1-j\leq w_2-k$. Define $P\in D_u^L(A)$ such that
$$\LT(P)=x_{m_1}\cdots x_{m_{j-1}} x_{n_{k}}\cdots x_{n_{k+w_1-j-1}} x_{m_{w_1}}\cdots x_{m_u}.$$ 
We claim that $P$ is well-defined and $d(P,N)<d(M,N)$.  
Obviously $P\neq M$.  

We verify that the variables in $\LT(P)$ are in NW-SE position in $A$.  
The fact that $(m_j,n_k,m_{w_1}=n_{w_2})$ is a violation and $k+w_1-j-1<w_2$ implies that $n_k,\hdots,n_{k+w_1-j-1}$ are in the same page.  So it suffices to show 
(i)  $m_{j-1}$ is NW of $n_k$ in $A$;
(ii) $n_{k+w_1-j-1}$ is NW of $m_{w_1}(=n_{w_2})$ in $A$. 
Indeed, (i) must hold because otherwise $(m_{j-1},n_k,m_t)$ is a violation of $(M,N)$, which contradicts (a).
For (ii), note that $k\le k+w_1-j-1<w_2$, so $n_{k+w_1-j-1}$ and $n_{w_2}$  lie on the same page and are in NW-SE position.  
	
Finally we verify that $d(P,N)<d(M,N)$.
Let $Z$ be the set of all incidences and let $|Z|=z$.  Then
	$d(M,N) = u+v-2z$, $d(P,N)  = u+v-2(z+w_1-j)=d(M,N)-2(w_1-j)<d(M,N)$, so $M,N$ are not $(A,B)$-adjacent.
\end{proof}    

\begin{prop} \label{chain''}
{\rm(a)}  There exist minors $M'\in D_u^L(A), N'\in D_v^L(B)$ such that the S-pair $S(M',N')$ may be expressed as a (possibly empty) finite sum $\sum a_iP_i$ such that for every $i$, $a_i$ is a monomial, $P_i\in D_u(A)\cup D_v(B)$ and $\LM(a_iP_i)<\LCM(\LM(M'),\LM(N'))$.

{\rm(b)}	There exists a chain of minors in $D^L_u(A)\cup D^L_v(B)$
	\[ M=M_0, M_1,\hdots ,M_{k-1},M_k=N  \] 
	such that for each $j=1,\hdots,k$, the S-pair $S(M_{j-1},M_j)$ may be expressed as a finite $\sum a_iP_i$ where $a_i$ are monomials, $P_i\in D_u(A)\cup D_v(B)$ and 
$\LT(a_iP_i)<\LCM(\LT(M_{j-1}),\LT(M_j))$ for all $i$.
\end{prop}

\begin{proof} 
Let $M'\in D_u^L(A),N'\in D_v^L(B)$ be such that $d(M',N')$ is minimal. 

(a) If $M'=N'$, the $S(M',N')=0$ and the statement is trivial. So we assume $M'\neq N'$. 
Then $M',N'$ are $(A,B)$-adjacent. 
By Proposition \ref{noviolation'} there is no violation of $(M',N')$.  The claim follows by Proposition \ref{lessthanL}.

(b) Let $L'=\LCM(\LM(M),\LM(M')),\quad L''=\LCM(\LM(N'),\LM(N))$. Then both $L'$ and $L''$ divide $L$.  Let $M=M_0,\hdots,M_h=M'$ and $N'=M_{h+1},\hdots,M_k=N$ be chains of minors from $D_u^{L'}(A)$ and $D_v^{L''}(B)$, respectively, given by Proposition \ref{thereexistsachain}.  
Note that $\LM(M_j)|L$ for all $j=0,\hdots,k$, so (b) follows from Propositions \ref{thereexistsachain} and (a).
\end{proof}

\begin{proof}[Proof of Theorem \ref{theorem:bipartite determinantal}]
If $M,N\in D_u(A)$ (or  $M,N\in D_v(B)$), then the theorem follows from Proposition \ref{thereexistsachain} and Proposition \ref{chain}. 
If $M\in D_u(A)$ and $N\in D_v(B)$ (or the other way around), then the theorem follows from  Proposition \ref{chain''} and Proposition \ref{chain}.  
\end{proof}

\section{Applications}

\subsection{Nakajima's affine graded quiver variety}\label{section:bipartite det}
In this subsection we explain the geometric motivation and application of bipartite determinantal ideals.

Nakajima's affine graded quiver variety $\mathfrak{M}_0^\bullet(V,W)$ (resp.~nonsingular graded quiver variety $\mathfrak{M}^\bullet(V,W)$)  is defined as certain affine algebro-geometric quotient $\mu^{-1}(0)//G_V$ (resp.~ GIT quotient $\mu^{-1}(0)^{\rm s}/G_V$). 
Recall the following definitions in \cite[\S4]{Nakajima} with slightly modified notations. Consider a bipartite quiver, that is, a quiver $\mathcal{Q}$ with vertex set ${\rm V}_\mathcal{Q}={\rm V}_{\rm source}\sqcup {\rm V}_{\rm sink}$ and each arrow $h$ has source ${\rm s}(h)\in {\rm V}_{\rm source}$ and target ${\rm t}(h)\in {\rm V}_{\rm sink}$. For each vertex $i \in {\rm V}_\mathcal{Q}$ attach a vector space $V_i$ of dimension $u_i$. Define a decorated quiver by adding a new vertex $i'$ for each $i\in {\rm V}_\mathcal{Q}$, adding an arrow $i'\to i$ if $i$ is a sink, adding an arrow $i\to i'$ if $i$ is a source (so the decorated quiver is still bipartite). For each arrow $i'\to i$ attach a vector space $W_i(1)$ to $i$ and $W_i(q^2)$ to $i'$; for each arrow $i\to i'$ attach $W_i(q^3)$ to $i$ and $W_i(q)$ to $i'$. Denote $\xi_i=0$ if $i\in {\rm V}_{\rm sink}$, $\xi_i=1$ if $i\in {\rm V}_{\rm source}$.  
The affine graded quiver variety $\mathfrak{M}_0^\bullet(V,W)$ is the image of $\mathfrak{M}^\bullet(V,W)$ under the natural projection and is isomorphic to ${\bf E}_{V,W}=\{ (\oplus {\bf x}_i,\oplus {\bf y}_h) \}$
where:

$\bullet$ ${\bf x}_i\in \Hom(W_i(q^{\xi_i+2}),W_i(q^{\xi_i}))$ for each $i\in {\rm V}$, ${\bf y}_h\in \Hom(W_{{\rm s}(h)}(q^3),W_{{\rm t}(h)}(1))$ for each arrow $h$; 

$\bullet$ $\dim\big({\rm Im }  {\bf x}_i + \sum_{i(h)=i} {\rm Im } {\bf y}_h\big)\le u_i$ for $i\in {\rm V}_{\rm sink}$, $\dim {\rm Im } \big({\bf x}_i\oplus \bigoplus_{o(h)=i}{\bf y}_h\big) \le u_i$ for $i\in  {\rm V}_{\rm source}$. 

In terms of matrices, the second condition can be written as 
``$\rank(A_\alpha)\le u_\alpha$ for $\alpha\in {\rm V}_{\rm sink}$, 
$\rank(A_\beta)\le u_\beta$ for $\beta\in {\rm V}_{\rm source}$''. 
So the affine graded quiver variety is defined by the bipartite determinantal ideal $I_{\mathcal{Q},{\bf m},{\bf u}}$. This is our motivation to study the latter. 

Nakajima's graded quiver varieties can be used in the study of finite-dimensional representations of the quantum affine algebras (\cite{Nakajima_JAMS}) and cluster algebras (see \cite{Nakajima, KQ, Qin, Li}). 
By studying the algebraic and combinatorial properties of the bipartite determinantal ideals, we can show that the corresponding Nakajima's affine graded quiver varieties are Cohen-Macaulay and compute their invariants, such as their Hilbert series (see \cite{Li3}).

\subsection{Double determinantal ideals}
For the rest of the paper, we fix the notation of double determinantal ideals as follows.
\begin{defn}\label{doubledeterminantalideal}.
	Let $m, n, r$ be positive integers, let $X= \{ X^{(k)}=[x^{(k)}_{ij}]\}_{k=1}^r $ be a set of variable matrices of dimension $m\times n$ (we call $i,j,k$ the row index, column index, and page index, respectively).  Define matrices 
$$ A= [ X^{(1)}|X^{(2)}|\hdots|X^{(r)} ], 
	\hspace{12pt}
	B= \left[
	\begin{array}{c}
	X^{(1)}  \\ \hline
	X^{(2)} \\ \hline
	\vdots \\ \hline
	X^{(r)}
	\end{array}
	\right] $$	
For positive integers  $u,v$,  $D_u(A)= \{u\text{-minors of }A\}$, $D_v(B)=\{v\text{-minors of }B\}$.  
 The double determinantal ideal $I_{m,n,u,v}^{(r)}$ is the ideal of $K[X]$ generated by $D_u(A)\cup D_v(B)$. 
 \end{defn}
\begin{exmp}\label{eg:m=n=r=u=v=2}
Take $m=n=r=u=v=2$, the matrices $A$, $B$, and the corresponding double determinantal ideal are as follows (where we use $a,b,\dots$ to denote $x^{(1)}_{11}, x^{(1)}_{12}, \dots$ for simplicity):
$$ A=\left[ \begin{matrix}x^{(1)}_{11}&x^{(1)}_{12}\\ x^{(1)}_{21}&x^{(1)}_{22}\end{matrix}\left|\,
\begin{matrix}x^{(2)}_{11}&x^{(2)}_{12}\\ x^{(2)}_{21}&x^{(2)}_{22}\end{matrix}  \right.\right]
=\left[ \begin{matrix}a&b\\ c&d\end{matrix}\left|\,
\begin{matrix}a'&b'\\ c'&d'\end{matrix}  \right.\right]
,
\quad
B=\begin{bmatrix}x^{(1)}_{11}&x^{(1)}_{12}\\ x^{(1)}_{21}&x^{(1)}_{22}\\
\hline
x^{(2)}_{11}&x^{(2)}_{12}\\ x^{(2)}_{21}&x^{(2)}_{22}
\end{bmatrix}, 
=\begin{bmatrix}a&b\\ c&d\\
\hline
a'&b'\\ c'&d'
\\
\end{bmatrix}, 
$$
$$\aligned
I_{2,2,2,2}^{(2)}=\langle
\begin{vmatrix}a&b\\ c&d\end{vmatrix},
\begin{vmatrix}a'&b'\\ c'&d'\end{vmatrix},
\begin{vmatrix}a&a'\\c&c'\end{vmatrix},
\begin{vmatrix}a&b'\\c&d'\end{vmatrix},
\begin{vmatrix}b&a'\\d&c'\end{vmatrix},
\begin{vmatrix}b&b'\\d&d'\end{vmatrix},
\begin{vmatrix}a&b\\a'&b'\end{vmatrix},
\begin{vmatrix}a&b\\c'&d'\end{vmatrix},
\begin{vmatrix}c&d\\a'&b'\end{vmatrix},
\begin{vmatrix}c&d\\c'&d'\end{vmatrix}
\rangle. 
\endaligned
$$
\end{exmp}

\subsection{Tensors}\label{tensors}

When studying the arrangement of the variables in our two matrices, it becomes clear that there is a higher order structure at play.  Each matrix may be thought of as a slice of a 3-dimensional array of variables (see Figure \ref{fig:slice}).  

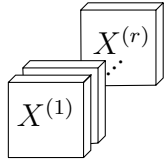
\begin{figure}[h!]
	\centering
	\begin{tikzpicture}
	\pgfmathsetmacro{\x}{1}
	\pgfmathsetmacro{\y}{.25}
	\pgfmathsetmacro{\z}{1}
	\path (0,0,\y) coordinate (A) (\x,0,\y) coordinate (B) (\x,0,0) coordinate (C) (0,0,0)
	coordinate (D) (0,\z,\y) coordinate (E) (\x,\z,\y) coordinate (F) (\x,\z,0) coordinate (G)
	(0,\z,0) coordinate (H);
	\draw (A)-- (B)-- (C)--(G)--(F)--(B) (A)-- (E)--(F)--(G)--(H)--(E);
	\draw (\x/2,.6*\z,\y)node{$X^{(1)}$};
	
	\path (.1*\x,\z,0) coordinate (I) (0,\z,-\y) coordinate (J) (0,\z,-2*\y) coordinate (K) (\x,\z,-2*\y) coordinate (L) (\x,0,-2*\y) coordinate (M) (\x,0,-\y) coordinate (N) (\x,.09*\z,0) coordinate (O) (\x,\z,-\y) coordinate (P)  ;
	\draw (I) -- (J) -- (K) -- (L) -- (M) -- (N) -- (O) (J) -- (P) -- (N) (L) -- (P);
	
	\path (.1*\x,.42*\z,0-8*\y ) coordinate (Q) (0,\z,-\y-8*\y) coordinate (R) (0,\z,-2*\y-8*\y) coordinate (S) (\x,\z,-2*\y-8*\y) coordinate (T) (\x,0,-2*\y-8*\y) coordinate (U) (\x,0,-\y-8*\y) coordinate (V) (.42*\x,.09*\z,0-8*\y) coordinate (W) (\x,\z,-\y-8*\y) coordinate (X)  ;
	\draw (Q) -- (R) -- (S) -- (T) -- (U) -- (V) -- (W) (R) -- (X) -- (V) (T) -- (X);
	\draw (.45*\x,.5*\z,-10*\y)node{$X^{(r)}$};
	
	\draw (.45*\x,.23*\z,-9*\y) node[rotate=45]{$...$};
	\end{tikzpicture}
	\caption{$\Scan(X)_3$ of 3-dimensional tensor $X$}
	\label{fig:slice}
\end{figure}

Thought of as a generalization of matrix, an $n$-dimensional array is called a tensor, and is a convenient way to represent multi-indexed data.  Applications are found in many fields of science, mathematics, and statistics.  
It is often beneficial to view tensors as multilinear maps on tensor products of vector spaces, and define them as objects that are independent of choice of bases, see \cite{Landsberg}.  However, for our purposes it suffices to assume that some basis is chosen, and define a tensor to be the representation of a multilinear map, given as an $n$-dimensional array of values in $K$.  
\begin{defn} [Definition 6.1.3 \cite{Bocci}]
	A tensor $X$ over $K$ of dimension (or order) $n$ of type (or size) 
	$a_1\times \hdots \times a_n$
	is a multidimensional table of elements of $K$, in which each element is determined by a multi-index $(i_1,\hdots,i_n)$ where each $i_j$ ranges from 1 to $a_j$.  Equivalently, $X$ is a multilinear map
	$ X:K^{a_1}\times \hdots \times K^{a_n}\rightarrow K$
where we consider the standard basis for each vector space $K^{a_i}$.
\end{defn}

When visualizing 3-dimensional tensors, we will adopt the convention that the first index increases downward along a vertical axis, the second increases across a horizontal axis, and the third increases backward along the third axis.  So a 3-dimensional tensor of variables of type $m\times n\times r$ will be given as follows.
\begin{center}
	\begin{tikzpicture}[scale=.7]
	\draw (0,0)node{$x_{111}$};
	\draw (2,0)node{$x_{1n1}$};
	
	\draw (0,-1.5)node{$x_{m11}$};
	\draw (2,-1.5)node{$x_{mn1}$};
	
	\draw[dashed] (0.5,0) to (2-.5,0);
	\draw[dashed] (0.5,-1.5) to (2-.5,-1.5);
	\draw[dashed] (0,0-.25) to (0,-1.5+.25);
	\draw[dashed] (2,0-.25) to (2,-1.5+.25);
	
	\draw (1.5,1)node{$x_{11r}$};
	\draw (3.5,1)node{$x_{1nr}$};
	\draw (3.5,-.5)node{$x_{mnr}$};
	
	\draw[dashed] (2,1) to (3,1);
	\draw[dashed] (3.5,1-.25) to (3.5,-.5+.25);
	\draw[dashed] (0+.25,0+.25) to (1.5-.35,1-.15);
	\draw[dashed] (2+.25,0+.25) to (3.5-.35,1-.15);
	\draw[dashed] (2+.25,-1.5+.25) to (3.5-.35,-.5-.15);
	
	\draw (-1.75,-.5) node{$X= $};
	\end{tikzpicture} 
\end{center}

\begin{defn}[Definition 8.2.1 \cite{Bocci}] \label{contraction}
	For an $n$-dimensional tensor $X=(x_{i_1,\hdots,i_n})$ of type $a_1\times \hdots \times a_n$, let $J\subseteq [n] = \{ 1,\hdots,n\}$ and set $\{ j_1,\hdots,j_q  \} = [n]\setminus J$.  For any choice $k_1,\hdots,k_q$ where $1\leq k_s \leq a_{j_s}$, define $X^J$, the $J$-contraction of $X$, to be the $q$-dimensional tensor of type $a_{j_1}\times \hdots \times a_{j_q}$ whose entry indexed by $(k_1,\hdots,k_q)$ is given by \[ x_{k_1,\hdots,k_q} = \sum x_{i_1,\hdots,i_n}   \]
	where the sum ranges over all entries where $i_{j_s}=k_s$ are fixed for all $s=1,\hdots ,q$.  
\end{defn}

\begin{exmp}\label{tensorex}
	Let $X$ be the following 3-dimensional tensor of type $2\times2\times2$:
	
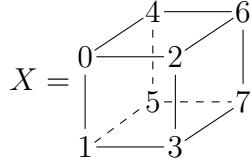
\begin{center}
		\begin{tikzpicture}[scale=.6]
		\draw (0,0)node{$0$};
		\draw (2,0)node{$2$};
		
		\draw (0,-2)node{$1$};
		\draw (2,-2)node{$3$};
		
		\draw (0.25,0) to (2-.25,0);
		\draw(0.25,-2) to (2-.25,-2);
		\draw(0,0-.4) to (0,-2+.4);
		\draw (2,0-.4) to (2,-2+.4);
		
		\draw (1.5,1)node{$4$};
		\draw (1.5,-1)node{$5$};
		
		\draw (3.5,1)node{$6$};
		\draw (3.5,-1)node{$7$};

		\draw (1.5+.25,1) to (3.5-.25,1);
		\draw (3.5,1-.4) to (3.5,-1+.4);
		
		\draw (0+.21,0+.14) to (1.5-.21,1-.14);
		\draw (2+.21,0+.14) to (3.5-.21,1-.14);
		\draw (2+.21,-2+.14) to (3.5-.21,-1-.14);
		
		\draw[dashed] (0+.21,-2+.1) to (1.5-.21,-1-.1);
		\draw[dashed] (1.5,1-.4) to (1.5,-1+.4);
		\draw[dashed] (1.5+.25,-1) to (3.5-.25,-1);
		
		\draw (-1,-.5)node{$X=$};	
		\end{tikzpicture}
	\end{center}
	and $J=\{ 2\}$, $L=\{ 3\}$, and $M=\{ 2,3 \}$.  Then the contractions along $J$ and $L$ are the matrices 
	\[ X^J = 
	\begin{pmatrix}
	2 & 10 \\
	4 & 12
	\end{pmatrix} , \hspace{1cm}   X^{L}=
	\begin{pmatrix}
	4 & 8 \\
	6 & 10
	\end{pmatrix} \]
	and the contraction along $M$ is the vector $ X^{M} = (12,16) $. 
\end{exmp}

\begin{rem}
	We note that $P=\frac{1}{28}X$ gives a joint probability distribution on three binomial random variables, since all entries of $P$ are non-negative and sum to 1.  Therefore, the contractions of a probability tensor are marginal distributions.  
\end{rem}

\begin{defn} [Definition 8.3.1 \cite{Bocci}] \label{scan}
	For an $n$-dimensional tensor $X=(x_{i_1,\hdots,i_n})$ of type $a_1\times \hdots \times a_n$, let $j\in  \{ 1,\hdots,n\}$.  Define $\Scan(X)_j$ to be the set of all $(n-1)$-dimensional tensors obtained by fixing the $j$-th index.
	\end{defn}

Referring to Example \ref{tensorex}, we have
$\Scan(X)_1 = \left \{
\begin{pmatrix}
0 & 2 \\
4 & 6
\end{pmatrix} ,
\begin{pmatrix}
1 & 3 \\ 
5 & 7
\end{pmatrix}  \right \}$,
$\Scan(X)_2 = \left \{  
\begin{pmatrix}
0 & 4 \\
1 & 5
\end{pmatrix},
\begin{pmatrix}
2 & 6 \\
3 & 7
\end{pmatrix} \right\}$, 
$\Scan(X)_3 =  \left\{
\begin{pmatrix}
0 & 2 \\
1 & 3
\end{pmatrix},
\begin{pmatrix}
4 & 6 \\
5 & 7
\end{pmatrix} \right\}$.
 
We make note of two facts.  First, that the contraction along a single axis is simply the sum of the tensors in the scan in the direction of that axis.  Therefore, in light of probability theory, the elements of the scan are scalar multiples of conditional distributions.  Secondly, the matrices in a scan of a 3-dimensional tensor of variables may be concatenated to create exactly the matrices $A$ and $B$ in Definition \ref{doubledeterminantalideal}.  This process is called flattening of a tensor, and it may be extended to higher dimensions, but for the purposes of this paper, we use the definition as applied to a 3-dimensional tensor. 

\begin{defn}[Definition 8.3.4 \cite{Bocci}]\label{flattening1}
	Let $X=(x_{ijk})$ be a 3-dimensional tensor of type $m\times n\times r$.  For each fixed $k=1,\hdots ,r$ let $X^{(k)}\in \Scan(X)_3$ be the matrix
	\[ X^{(k)}=
	\begin{pmatrix}
	x_{11k} & \hdots & x_{1nk} \\
	\vdots & & \vdots \\
	x_{m1k} & \hdots &  x_{mnk}
	\end{pmatrix}.\]
	Define the flattenings by concatenating the matrices of $\Scan(X)_3$ as follows.
	\[ F_1(X) = [ X^{(1)} | \hdots | X^{(r)} ], \hspace{1cm}  F_2(X) = [ (X^{(1)})^T | \hdots | (X^{(r)})^T ] \]
\end{defn}

One can see that Definition \ref{doubledeterminantalideal} may be altered for slightly more generality.  By endowing $X$ with the structure of a 3-dimensional tensor, we see that $F_1(X)=A$ and $F_2(X)=B^T$.  
Therefore, the double determinantal ideal is generated by minors of two particular flattenings of a 3-dimensional tensor of variables.  We will consider flattenings and their applications in more detail in \S\ref{tripledeterminantalideal} and \S\ref{algebraicstatistics}.  In particular, we see that our problem focuses on two of the three natural ways to flatten a 3-dimensional tensor, and it would be natural to ask whether our proof might extend to include the third flattenings of $X$, as well.  

\subsection{Generalizations}\label{tripledeterminantalideal}
In this subsection we consider how the main result fits in the more general context of tensors.  A flattening is just one special kind of tensor reshaping, which can be thought of as the removal of some of the structure of the original tensor.  Flattening is a particularly useful reshaping, since matrices are easy to express visually.  The smallest internal sub-structure of a tensor consists of the 1-dimensional arrays (vectors), called fibers.  A 3-dimensional tensor has three ways to slice into fibers, each running along the direction of an axis, we may call them column fibers, row fibers, and tube fibers (for a picture, see \cite[Figure 2.3.1]{Landsberg}).  
We may classify flattenings of a 3-dimensional tensor $X$ of type $m\times n \times r$ into three categories of matrices, ones in which the columns are defined by the column fibers, the row fibers, or the tube fibers. Some generic examples might look like the following. 
\[ X_1 = \begin{bmatrix}
x_{111} & \hdots & x_{1nr} \\
\vdots & & \vdots \\
x_{m11} & \hdots & x_{mnr}
\end{bmatrix},\quad
X_2 =  \begin{bmatrix}
x_{111} & \hdots & x_{m1r} \\
\vdots & & \vdots \\
x_{1n1} & \hdots & x_{mnr}
\end{bmatrix},\quad 
X_3 =  \begin{bmatrix}
x_{111} & \hdots & x_{mn1} \\
\vdots & & \vdots \\
x_{11r} & \hdots & x_{mnr}
\end{bmatrix} \]
Since we are interested in the ideals generated by the minors of these matrices, we may arbitrarily choose the order in which to write the columns.  
Therefore, we may categorize flattenings into three classes of matrices, with each class containing all matrices whose rows are indexed by one specified index.   
\begin{defn}
	Let $X=(x_{i_1,\hdots,i_n})$ be an $n$-dimensional tensor of variables of size $a_1\times \hdots \times a_n$.  
	For $1\le j\le n$, 
	define the $j^{th}$ \textit{flattening class} $\overline{X}_j$ to be the set of all $a_j\times(a_1\hdots a_{j-1}a_{j+1}\hdots a_n) $ matrices containing all variables $x_{1,\hdots,1},\hdots,x_{a_1,\hdots,a_n}$, in which each column is of the following form:
	$$\begin{bmatrix} x_{i_1,\dots,1,\dots,i_n}\\x_{i_1,\dots,2,\dots,i_n}\\ \vdots \\ x_{i_1,\dots,a_j,\dots,i_n}\end{bmatrix}$$
that is, the variables in a column have the same fixed indices except for the $j^{th}$ index.  
      We call each matrix in the $j^{th}$ flattening class a $j^{th}$ \textit{flattening matrix}.  
      So any two $j^{th}$ flattening matrices are the same up to permuting columns. 
 \end{defn}

Theorem \ref{theorem:bipartite determinantal} implies the following. For any 3-dimensional tensor of variables $X$ of type $m\times n\times r$, there exist $X_1\in \overline{X}_1$ and $X_2\in \overline{X}_2$ which are block matrices of the form 
$X_1 = [Y_1|\cdots | Y_r]$, $X_2= \left[ Y_1^T|\cdots | Y_r^T \right]$
with $Y_1,\hdots,Y_{r}\in \Scan(X)_3$, such that $D_u(X_1)\cup D_v(X_2)$ forms a Gr\"obner basis, with respect to a monomial order that is consistent with both $X_1$ and $X_2$.  More generally, we have the following.

\begin{cor} \label{generalize}
	Let $X$ be an $n$-dimensional tensor of variables of type $a_1\times \hdots \times a_n$.  Then for any $j,k\in \{ 1,\hdots,n\} $ and any positive integers   $u_j, u_k$, there exists a $j^{th}$ flattening matrix $X_j$ and a $k^{th}$ flattening matrix $X_k$ such that the minors $D_{u_j}(X_j)\cup D_{u_k}(X_k)$ form a Gr\"obner basis with respect to an appropriate monomial order.
\end{cor}

\begin{proof}
By permuting the indices, we may assume that $j=1$ and $k=2$.  
Write the $C=\{(i_3,\dots,i_n):1\le i_t\le a_t \textrm{ for } t=3,\dots,n\}$ as an ordered list of $r=\prod_{t=3}^n a_t$ tuples, and for $1\le i\le r$ define 	
	\[ Y_i= \begin{bmatrix}
	y_{11}^{(i)}& \hdots & y_{1a_2}^{(i)}\\
	\vdots & & \vdots \\
	y_{a_11}^{(i)}& \hdots & y_{a_1a_2}^{(i)}
	\end{bmatrix} 
	:=  \begin{bmatrix}
	x_{11b_3\hdots b_n} & \hdots & x_{1a_2b_3\hdots b_n} \\
	\vdots & & \vdots \\
	x_{a_11b_3\hdots b_n} & \hdots & x_{a_1a_2b_3\hdots b_n}
	\end{bmatrix} \]
where $(b_3,\hdots,b_n)$ is the $i$-th element of $C$.	
Let
$X_1 = [Y_1|\hdots | Y_r]$,  $X_2= \left[ Y_1^T|\hdots | Y_r^T \right]$. 
Let  ``$>$''  be the monomial order on the variables of $X$ induced by
	\[ y_{j_1k_1}^{(i_1)}> y_{j_2k_2}^{(i_2)} \Leftrightarrow (i_1<i_2)\text{ or } (i_1=i_2,j_1<j_2)\text{ or } (i_1=i_2,j_1=j_2,k_1<k_2).\]  
	Then the corollary follows from Theorem \ref{theorem:bipartite determinantal}.
\end{proof}

\begin{defn} \cite[Definition 3.4.1]{Landsberg}
	The \textit{subspace variety} $\Sub_{b_1,\hdots,b_n}(K^{a_1}\otimes \hdots \otimes K^{a_n}) 
	:= \{ T\in K^{a_1}\otimes \hdots \otimes K^{a_n} | \dim T((K^{a_i})^*)\leq b_i,  \forall 1\leq i\leq n \}  $ is the common zero set of all $(b_i+1)$-minors of the $i^{th}$ matrix flattening of $X$. 
\end{defn}

\begin{cor}
	For any $n$-dimensional tensor of variables of size $a_1\times \hdots \times a_n$, a multilinear rank vector $(b_1,\hdots,b_n)$, and two integers $i,j\in\{1,\dots,n\}$, if $b_i\ge a_i$ for $i\neq j,k$, then the subspace variety $\Sub_{b_1,\hdots,b_n}(K^{a_1}\otimes \hdots \otimes K^{a_n})$ is defined by the prime ideal whose generators, the $(b_j+1)$-minors of $X_j$ and the $(b_k+1)$-minors of $X_k$, can be shown to form a Gr\"obner basis under an appropriate monomial order.  
\end{cor}

\begin{proof}
For $i\neq j,k$, since $b_i\ge a_i$, there are no $(b_i+1)$-minors in the $i^{th}$ matrix flattening of $X$. So the statement follows from Corollary \ref{generalize}.
\end{proof}

We now consider a different generalization of our result, in the context of tensors.  Since there are exactly three flattening classes, it is natural to consider including the minors of the third flattening class.  
\begin{defn}
	Let $X$ be a 3-dimensional tensor of variables of size $m\times n\times r$.  For any integers $1< u\leq \min(m,rn)$, $1 < v \leq \min(n,mr)$, and $1 < w\leq \min(r,mn)$, let $I_{u,v}$ be the double determinantal ideal generated by the $u$-minors of a $1^{st}$ flattening matrix and the $v$-minors of a  $2^{nd}$ flattening matrix.  Let the \textit{triple determinantal ideal}, $I_{u,v,w}$, be defined by the generators of $I_{u,v}$ together with the $w$-minors of a $3^{rd}$ flattening matrix.  
\end{defn}

\begin{prop}\label{double=triple}
	For any 3-dimensional tensor of variables of type $m\times n\times r$ and any positive integers $2\leq u\leq \min(m,rn)$, $2\leq v \leq \min(n,mr)$, and $2\leq w\leq \min(r,mn)$, we have  
	\[ I_{u,v}=I_{u,v,w} \Leftrightarrow (u-1)(v-1) \leq w-1. \]
	As a consequence, if the largest number of $u-1, v-1, w-1$ is not less than the product of the other two numbers, then the triple determinantal ideal $I_{u,v,w}$ is a double determinantal ideal, thus satisfies all the properties of the latter.
\end{prop}

\begin{proof}
	``$\Rightarrow$'': Suppose $ (u-1)(v-1) >w-1$.  We show that $I_{u,v}\subsetneq I_{u,v,w}$ by constructing a point $T\in V(I_{u,v})\setminus V(I_{u,v,w})$. Indeed, let $T=(t_{ijk})$ be defined by 
	\[ t_{ijk} = \begin{cases} 1 & \text{if } 1\leq i\leq u-1, 1\leq j\leq v-1, k=(i-1)(v-1)+j; \\ 0 & \text{otherwise.} \end{cases}  \]
Then $\rank\ T_1=u-1$, $\rank\ T_2=v-1$, $\rank\ T_3=\min\{(u-1)(v-1),r\}>w-1$. 
		
	``$\Leftarrow$'': Suppose $ (u-1)(v-1) \leq w-1$.  Obviously $I_{u,v}\subseteq I_{u,v,w}$, so it suffices to prove  $I_{u,v}\supseteq I_{u,v,w}$.  We may assume that $K$ is an algebraically closed field.  This is because equality of ideals in a polynomial ring is not sensitive to field extensions by Lemma \ref{extensionfieldlem}.  
	
We claim that $V(I_{u,v})=V(I_{u,v,w})$. It suffices to prove ``$\subseteq$''. Let $T=(t_{ijk})\in V(I_{u,v})$.  Without loss of generality, we may assume that $T_1$ is a first flattening matrix with rank $u'-1$ (with $u'\le u$) whose top $u'-1$ rows are linearly independent, and $T_2$ is a second flattening matrix of rank $v'-1$ (with $v'\le v$) whose top $v'-1$ rows are linearly independent. 
Then $T_1$ can be factored as $T_1=AR_1$ 
where $A=(a_{ij})$ is a $m\times(u'-1)$ matrix whose top $u'-1$ rows form an identity matrix, and $R_1\in \Mat_{(u'-1)\times nr}(K)$ is the submatrix of $T_1$ that has the top $(u'-1)$ rows of $T_1$; 
similarly, $T_2=BR_2$ 
where $B=(b_{ij})$ is a $n\times(v'-1)$ matrix whose top is an identity matrix and $R_2\in \Mat_{(v'-1)\times mr}(K)$  is the submatrix of $T_2$ that has the top $(v'-1)$ rows of $T_2$; 
Then
$t_{ijk}=\sum_{p=1}^{u'-1} a_{ip}t_{pjk}=\sum_{q=1}^{v'-1}b_{jq}t_{iqk}$. 
Denote ${\bf t}_{ij*}=[t_{ij1},\dots,t_{ijr}]^T\in K^r$. 
Then 
${\bf t}_{ij*}=\sum_{p=1}^{u'-1}\sum_{q=1}^{v'-1} a_{ip}b_{jq} {\bf t}_{pq*}$.
So a third flattening matrix $T_3$ must have rank $\le (u'-1)(v'-1)\le (u-1)(v-1)\le w-1$, thus $T\in V(I_{u,v,w})$.

Now by the strong Nullstellensatz, $\sqrt{I_{u,v}}=\sqrt{I_{u,v,w}}$.  Then by Corollary \ref{generalize}, the defining minors of $I_{u,v}$ form a Gr\"obner basis, and so $\LT(I_{u,v})$ is square-free, and thus, $I_{u,v}$ is a radical ideal, implying that $I_{u,v}=\sqrt{I_{u,v}}=\sqrt{I_{u,v,w}}\supseteq I_{u,v,w}$.   
\end{proof}

The above proof requires the following technical, but well known result. 

\begin{lem}\label{extensionfieldlem}
	If $k$ is a subfield of $F$, $R=k[x_1,...,x_n]$, $S=F [x_1,...,x_n]$, and $I$ and $J$ be ideals of $R$.  If $IS=JS$, then $I=J$.
\end{lem}

\begin{proof}
	Since $k$ is a field, the only maximal ideal, $(0)$, in $k$ satisfies $(0)^e=(0)$, which is not $(1)$ in $F$, so by \cite[p.45 exercise 16]{Atiyah} $F$ is faithfully flat over $k$.  Then, $S$ $(= F \otimes_k R)$ is faithfully flat over $R$ by \cite[p.46]{Matsumura}.  So, $IS \cap R=I$ by \cite[Thm 7.5]{Matsumura} (and, likewise $JS \cap R=J$).  Since $IS=JS$, we have that $I = IS\cap R = JS\cap R = J$.
\end{proof}

\subsection{Algebraic Statistics}\label{algebraicstatistics}

The use of algebraic and geometric techniques to study statistical problems is relatively recent \cite{Drton}, but provides a rich new context for interpreting classical computational algebra problems, such as ours.  In this section we describe how to interpret the double determinantal variety from a statistical viewpoint.  We will restrict our discussion to discrete random variables, but the concepts may extend to continuous random variables.

\subsection{Independence Models}
Let $X$ and $Y$ be discrete random variables on $m$ and $n$ states, respectively.  We consider whether $X$ and $Y$ are independent.  Denote their joint probability density as $P(X=i,Y=j)=p_{ij}$, and denote their marginal probabilities as $P(X=i)=p_{i+}=\sum_{j=1}^n p_{ij}$ and $P(Y=j)=p_{+j}=\sum_{i=1}^m p_{ij}$.  Probability theory dictates that for independence, the probabilities must satisfy the equation $P(X=i,Y=j)=P(X=i) P(Y=j)$, or equivalently, $p_{ij}=p_{i+} p_{+j}$, for all $i,j$.  Therefore, $X$ is independent of $Y$ (denoted $X\ind Y$) if and only if for every $i,j,k,l$, 
\[ p_{ij}p_{kl}= ( p_{i+}p_{+j} )( p_{k+}p_{+l} ) = ( p_{i+} p_{+l} )( p_{k+}p_{+j} ) = p_{il}p_{kj} \; (\textrm{or equivalently},
\begin{vmatrix}
p_{ij} & p_{il} \\
p_{kj} & p_{kl}
\end{vmatrix} =0). \]
This statement is equivalent to saying that the matrix $P=(p_{ij})$ has rank at most one.

A discrete probability distribution is a collection of non-negative values that sum to one, so we may think of our joint probability, geometrically, as a point of $\mathbb{R}^{mn}$ that is contained in the probability simplex
$\Delta_{mn-1} = \{ (p_{ij}) \in \mathbb{R}^{mn}: p_{ij}\geq 0, \sum_{ij} p_{ij} =1 \}$. 
The set of all distributions for two independent discrete random variables is called the independence model, which is the intersection of the probability simplex with the Segre variety  \cite{Drton},
\begin{equation} \label{indmodel}
\aligned
\mathcal{M}_{X\ind Y}&= \{ P=(p_{ij})\in \Delta_{mn-1}: \rank(P)\leq 1 \} \\
&=
\{ P=(p_{ij}) \in \mathbb{R}^{mn}: \rank(P) \leq 1 \} \cap \Delta_{mn-1}.
\endaligned
\end{equation}  
The ideal corresponding to the variety  $\{ P=(p_{ij}) \in \mathbb{R}^{mn}: \rank(P) \leq 1 \}$ is the independence ideal, $I_{X\ind Y} \subseteq K[p_{ij}]$, which is generated by all 2-minors of the matrix of variables $(p_{ij})$.

Inferential statistics uses a randomly collected set of data to decide something about the unknown distribution of the random variables observed.  The observation of data is the realization of a geometric point, and statisticians develop means to decide whether that point is close enough to a particular variety, or set of points that satisfy a certain notion, such as independence.  The dimension of the tensor required to store the data equals the number of random variables observed.  Suppose we are given $t$ discrete random variables, and two non-intersecting subsets, $S_1$ and $S_2$, of $\{ 1,\hdots , t\}$.  Then, an independence statement asserts that the outcomes of all variables corresponding to $S_1$ are independent of the outcomes of all variables corresponding to $S_2$, in that the joint density is equal to the product of the respective marginal densities.  As such, each independence statement corresponds to the requirement that some matrix (a contraction or a flattening of the probability tensor) has rank $\leq 1$.  

We specify to an order 3 tensor.  Let $X_1,X_2,X_3$ be discrete random variables on $m,n,r$ states respectively.  We denote the joint probability distribution using the values $p_{ijk}=P(X_1=i,X_2=j,X_3=k)$, and we have the following probability tensor.  
\begin{center}
	\begin{tikzpicture}[scale=.6]
	\draw (0,0)node{$p_{111}$};
	\draw (2,0)node{$p_{1n1}$};
	
	\draw (0,-2)node{$p_{m11}$};
	\draw (2,-2)node{$p_{mn1}$};
	
	\draw[dashed] (0.6,0) to (2-.6,0);
	\draw[dashed] (0.6,-2) to (2-.6,-2);
	\draw[dashed] (0,0-.4) to (0,-2+.4);
	\draw[dashed] (2,0-.4) to (2,-2+.4);
	
	\draw (1.5,1)node{$p_{11r}$};
	\draw (3.5,1)node{$p_{1nr}$};
	\draw (3.5,-1)node{$p_{mnr}$};
	
	\draw[dashed] (2.15,1) to (3-.15,1);
	\draw[dashed] (3.5,1-.35) to (3.5,-1+.35);
	\draw[dashed] (0+.35,0+.35) to (1.5-.35,1-.25);
	\draw[dashed] (2+.35,0+.35) to (3.5-.35,1-.25);
	\draw[dashed] (2+.35,-2+.35) to (3.5-.35,-1-.25);
	
	\draw (-1.25,-.75)node{$P=$};
	\end{tikzpicture}
\end{center}
Then, the marginal distributions can be ascertained from a contraction (see Definition \ref{contraction}), $P^J$, along a subset of indices $J\subset \{ 1,2,3 \}$.  For example, the joint marginal distribution for $X_1$ and $X_2$ is the contraction of $P$ along the set $\{3\}$, resulting in the following matrix
\[ P^{\{ 3 \} } = 
\begin{pmatrix}
p_{11+} & \hdots & p_{1n+} \\
\vdots & & \vdots \\
p_{m1+} & \hdots & p_{mn+}
\end{pmatrix}\]
where $p_{ij+}=P(X_1=i,X_2=j)=\sum_{k} p_{ijk}$.  The marginal distribution for $X_1$, for instance, would be the $\{2,3\}$-contraction  
\[ P^{\{2,3 \}}  = (p_{1++},\hdots, p_{m++}) \]
where $P(X=i)=p_{i++}=\sum_{j,k} p_{ijk}$.  We consider the two possible types of independence statements for three random variables:  $X_i\ind X_j$ results in a rank bound of 1 on a contraction matrix, and $X_i\ind (X_j,X_k)$ results in a rank bound of 1 on a flattening matrix.

\begin{defn}\label{ideals1}
	For any discrete random variables $X_1,X_2,X_3$ on $m,n,r$ states respectively, let the joint probability distribution given by the order 3 tensor $P=(p_{ijk})$ where $p_{ijk}=P(X_1=i,X_2=j,X_3=k)$.  Let $P_1$,$P_2$, and $P_3$ be a first, second, and third flattening matrix, respectively.  Then, for distinct indices $a,b,c\in \{1,2,3 \}$ we have the following independence models:
	\[  \mathcal{M}_{X_a \ind (X_b,X_c)} = \{P\in \Delta_{mnr-1}: \rank(P_a)\leq 1   \} \]
	\[   \mathcal{M}_{X_a \ind X_b} = \{ P\in \Delta_{mnr-1} : \rank(P^{\{ c \}})\leq 1 \} \]
\end{defn}

In general, each set of independence statements, $\mathcal{C}$, corresponds to an ideal, $I_{\mathcal{C}}$, in the polynomial ring in $mnr$ variables $R=K[p_{ijk}]$, generated by the 2-minors of the corresponding matrices (given by appropriate contractions or flattenings) \cite{Sullivant}.  Under certain conditions, the independence ideal is the double determinantal ideal.

\begin{cor}\label{rem2} 
	Let $X_1,\hdots, X_n$ be discrete random variables on $a_1,\hdots, a_n$ states, respectively.  Let $P$ be the $n$-dimensional probability tensor.  Then for any $i,j\in \{1,\hdots,n \}$
	and independence set
	$\mathcal{C}_{i,j}=\{ X_i\ind (X_1,\hdots,X_{i-1},X_{i+1},\hdots,X_n), X_j\ind (X_1,\hdots,X_{i-j},X_{j+1},\hdots,X_n) \}$,  
	the independence ideal is the double determinantal ideal
	($I_{\mathcal{C}_{i,j}} = I_{2,2}$)
	and is generated by the 2-minors of two flattening matrices, which form a Gr\"obner basis under an appropriate monomial order.
\end{cor}

\begin{proof}
	Follows from Corollary \ref{generalize}.
\end{proof}

\subsection{Conditional Independence and Hidden Variables}

To interpret the problem with larger rank bounds, we may consider conditional independence and hidden variables.  For conditional independence statements, we use the standard result that the probability of event $A$ occurring, given event $B$ has occurred, is $P(A|B)= {P(A \text{ and } B)}/{P(B)}$.  So, in our setting of three discrete random variables $X_1$, $X_2$ and $X_3$, on $m$, $n$, and $r$ states, respectively, we have that, for for any $1\leq k\leq r$, $P(X_1=i,X_2=j|X_3=k) = {p_{ijk}}/{p_{++k}}$.  Therefore, the joint density of $X_1$ and $X_2$ given $X_3=k$ is 
$ \dfrac{1}{p_{++k}}
\begin{pmatrix}
p_{11k}&\hdots & p_{1nk} \\
\vdots & & \vdots \\
p_{m1k} & \hdots & p_{mnk}
\end{pmatrix}
$
where the matrix is in $Scan(P)_3$ (see definition \ref{scan}). 
Imposing an independence requirement on a set of conditioned variables, puts a rank bound on each of the conditional distribution matrices.   

\begin{defn}
	For three discrete random variables with joint probability tensor $P$, and for any distinct indices $a,b,c\in \{ 1,2,3\}$, the conditional independence model is 
	\[ \mathcal{M}_{X_a\ind X_b|X_c}=\{ P\in \Delta_{mnr-1}: \rank(P')\leq 1, \textrm{ for any }P'\in Scan(P)_c    \} .  \]
\end{defn}

A hidden variable is one whose outcomes are not observed, for whatever reason. When conducting an experiment, one might suspect that there is some hidden variable that dictates the distribution of probabilities for the observable outcomes.  That is, suppose that $P=(p_{ijk})$ is a 3-dimensional probability tensor for the joint distribution of $X_1,X_2,X_3$, and suppose there is an unobservable random variable $Y$ with $u$ states.  Let $P(Y=i)=\pi_i$ for each $i$. For each $l=1,\hdots,u$, given that $Y=l$, let the conditional joint distribution on $X_1,X_2,X_3$ be $P^{(l)} = (p_{ijk}^{(l)})\in \mathcal{M}$, from a fixed model $\mathcal{M}$. 
Then
\[  p_{ijk}=\sum_{l=1}^u P(X_1=i,X_2=j,X_3=k|Y=l)P(Y=l) = \sum_{l=1}^u p_{ijk}^{(l)}\pi_l. \]
Since the conditional distributions all come from the same model  $\mathcal{M}$, then $P$ is a convex combination of points of $\mathcal{M}$.  That is, it is in the $u^{th}$ mixture of the model $\mathcal{M}$ \cite{Sullivant},
\[ P \in  \Mixt^{u}(\mathcal{M}) := \left \{ \sum_{i=1}^u \pi_i P_i : \pi_i\geq 0, \sum_i\pi_i=1,P_i\in \mathcal{M}   \right \} .  \]
Mixture models are contained in the intersection of secant varieties with $\Delta_{mnr-1}$, but are not, in general, equal \cite[Example 14.1.6]{Sullivant}.  In fact, when rank is $\geq 2$ mixture models may have complicated  boundaries and require the study of broader notions of rank \cite{Allman}.  For the purposes of this paper, we aim to provide a context for interpreting the double determinantal ideal.  
From this standpoint, suppose that we would like to test some conditional independence constraints with a hidden variable in mind.  For instance, the constraint $X_1\ind (X_2,X_3)|Y $ would impose rank $\le 1$ condition on the first flattening of each of the $u$ conditional probability tensors $P^{(1)},\hdots,P^{(u)}$.  This would imply that the probability tensor $P$ for $X_1,X_2,X_3$ (unconditioned on $Y$) would have a first flattening which may be written as a convex combination of the $u$ first flattenings of $P^{(1)},\hdots,P^{(u)}$, each having rank $\le1$.  As such, the first flattening of the tensor $P$ would have rank $\le u$.  Therefore, the set of all probability tensors 
\[ \{ P\in \mathbb{R}^3: \rank P_1\leq u      \}    \]
which have first flattening with rank at most $u$ would be a variety whose intersection with the probability simplex would constitute the model for the independence condition $X_1\ind (X_2,X_3)|Y$ when conditioned on a hidden variable $Y$ on $u$ states.  And, the corresponding conditional independence ideal is generated by all the size $u+1$ minors of the first flattening of a generic matrix.   

\begin{rem}
	Let $X_1$, $X_2$, and $X_3$ be discrete random variables on $m$, $n$, and $r$ states respectively, and let $Y_1$, $Y_2$, and $Y_3$ be hidden variables on $u$, $v$, and $w$ states respectively.
	The conditional independence ideal corresponding to the independence statements 
	$$\mathcal{C}_2=\{ X_1\ind (X_2,X_3)|Y_1, X_2\ind (X_1,X_3)|Y_2\}$$ is given by the double determinantal ideal
	$I_{\mathcal{C}_2} = I_{u,v}$
	which is generated by minors (which form a Gr\"obner basis as we have shown).  The conditional independence ideal corresponding to the independence statements \[ \mathcal{C}_3=\{ X_1\ind (X_2,X_3)|Y_1, X_2\ind (X_1,X_3)|Y_2,X_3\ind(X_1,X_2)|Y_3\} \] 
	is given by the triple determinantal ideal
	$I_{\mathcal{C}_3} = I_{u,v,w}$.
	As a result of Proposition \ref{double=triple}, these ideals coincide whenever $(u-1)(v-1)\leq w-1$. This gives a condition under which the independence constraints of $\mathcal{C}_2$ imply that $X_3\ind(X_1,X_2)|Y_3$.  
\end{rem}

\end{document}